\documentclass[a4paper, 10pt, reqno]{amsart}

\usepackage[utf8]{inputenc}
\usepackage[T1]{fontenc}
\usepackage{amsmath}
\usepackage{amsthm}
\usepackage{amssymb}
\usepackage[inline]{enumitem}
\usepackage[dvips]{graphicx}
\usepackage{color}
\usepackage{comment}
\usepackage{hyperref}
\usepackage{url}
\usepackage{stackrel}
\usepackage{fancyhdr}
\usepackage{mathrsfs}
\usepackage{mathtools}
\usepackage{stmaryrd}

\newtheorem{theorem}{Theorem}[section]
\newtheorem*{main*}{Main Theorem}
\newtheorem{lemma}[theorem]{Lemma}
\newtheorem{proposition}[theorem]{Proposition}
\newtheorem{corollary}[theorem]{Corollary}

\newtheorem*{claim*}{\textsc{Claim}}

\theoremstyle{definition}

\newtheorem*{def*}{Definition}
\newtheorem{definition}[theorem]{Definition}
\newtheorem{example}[theorem]{Example}
\newtheorem{examples}[theorem]{Examples}

\newtheorem{remark}[theorem]{\textbf{Remark}}

\definecolor{blue-url}{RGB}{0,0,100}
\definecolor{green-link}{RGB}{0,97,0}

\hypersetup{
    pdftitle={Structural properties of subadditive families},
    pdfauthor={Salvatore Tringali},
    pdfmenubar=false,
    pdffitwindow=true,
    pdfstartview=FitH,
    colorlinks=true,
    linkcolor=blue-url,
    citecolor=green-link,
    urlcolor=blue-url
}

\DeclareMathSymbol{\widehatsym}{\mathord}{largesymbols}{"62}

\renewcommand{\emptyset}{\varnothing}

\renewcommand{\setminus}{\smallsetminus}

\providecommand{\eps}{\varepsilon}
\providecommand{\LLc}{\mathscr{L}}
\providecommand{\LLs}{\mathsf{L}}

\providecommand{\NNb}{{\mathbf{N}}}

\providecommand{\PPc}{\mathcal{P}}

\providecommand{\UUc}{\mathscr{U}}

\DeclareMathOperator{\lcm}{lcm}

\providecommand\llb{\llbracket}
\providecommand\rrb{\rrbracket}

\newcommand\DeltaF[1]{\Delta(#1)}

\newcommand{\fixed}[2][1]{%
  \begingroup
  \spaceskip=#1\fontdimen2\font minus \fontdimen4\font
  \xspaceskip=0pt\relax
  #2%
  \endgroup
}

\begin{document}
\title{Structural properties of subadditive families \\ with applications to factorization theory}
\author{Salvatore Tringali}
\address{Institute for Mathematics and Scientific Computing, University of Graz, NAWI Graz | Heinrichstr. 36, 8010 Graz, Austria}
\email{salvatore.tringali@uni-graz.at}
\urladdr{\href{https://imsc.uni-graz.at/tringali/}{https://imsc.uni-graz.at/tringali/}}
%
%
%
%
%

\keywords{Accepted elasticity; maximal orders; non-unique factorization; periodicity; structure theorems; transfer Krull monoids (and domains); unions of sets of lengths; subadditive families; weakly Krull monoids.}
\thanks{The author was supported by the Austrian Science Fund (FWF), Project No. M 1900-N39.}
\begin{abstract}
\noindent{}Let $H$ be a multiplicatively written monoid.
Given $k \in \mathbf N^+$, we denote by $\mathscr U_k$ the set of all $\ell \in \mathbf N^+$ such that $a_1 \cdots a_k = b_1 \cdots b_\ell$ for some atoms (or irreducible elements) $a_1, \ldots, a_k, b_1, \ldots, b_\ell \in H$. The sets $\mathscr U_k$ are one of the most fundamental invariants studied in the theory of non-unique factorization, and understanding their structure is a basic problem in the field: In particular, it is known that, in many cases of interest, these sets are almost arithmetic progressions with the same difference and bound for all large $k$, which is usually expressed by saying that $H$ satisfies the Structure Theorem for Unions.
The present paper improves the current state of the art on this problem.

More precisely, we will show that, under mild assumptions on $H$, not only does the Structure Theorem for Unions hold, but there also exists $\mu \in \mathbf N^+$ such that, for every $M \in \mathbf N$, the sequences
$$
\bigl((\mathscr U_k - \inf \mathscr U_k) \cap \llb 0, M \rrb\bigr)_{k \ge 1}
\quad\text{and}\quad
\bigl((\sup \mathscr U_k - \mathscr U_k) \cap \llb 0, M \rrb \bigr)_{k \ge 1}
$$
are $\mu$-periodic from some point on.
The result applies, for instance, to (the multiplicative monoid of) all commutative Krull do\-mains (e.g., Dedekind domains) with finite class group;
a variety of weakly Krull commutative domains (including all orders in number fields with finite elasticity); some maximal orders in central simple algebras over global fields; and all numerical monoids.

Large parts of the proofs are worked out in a ``purely additive model'' (where no explicit reference to monoids or atoms is ever made), by inquiring into the properties of what we call a subadditive family, i.e., a collection $\mathscr L$ of subsets of $\mathbf N$ such that, for all $L_1, L_2 \in \mathscr L$, there is $L \in \mathscr L$ with $L_1 + L_2 \subseteq L$.
\end{abstract}
\maketitle
\thispagestyle{empty}

\section{Introduction}
\label{sec:intro}
Similar to factorizations in the integers, non-zero non-unit elements in many integral domains can be written as (finite) products of irreducible elements, but unlike the case of the integers, such factorizations need not be
essentially unique: The main goal of factorization theory is to study phenomena arising from this lack of uniqueness and to classify them by an assortment of invariants.

The subject developed out of algebraic number theory, and a turning point in its history has been the crucial observation, which can be traced back to the early work of F.~Halter-Koch and A.~Geroldinger in the area, that questions of non-unique factorization in integral domains are purely multiplicative in nature and, hence, can be conveniently rephrased in the language of monoids, with the latter providing ``canonical models'' of the phenomena under consideration that would not be available otherwise \cite{GeHK06}. It is, however, only in recent years that fundamental aspects of factorization theory have been systematically extended to non-commutative or non-cancellative settings, see \cite{ChFoGeOb16, Ge16c, FTr} and references therein. Notably, an impetus to these developments has come from a more profound comprehension of the interplay between factorization theory and arithmetic combinatorics, which is also the leitmotif of this paper.

To begin, let $H$ be a multiplicatively written monoid
(basic notations and terminology will be explained later).
We take $\UUc_0 := \{0\} \subseteq \mathbf N$, and given $k \in \mathbf N^+$, we denote by $\mathscr U_k(H)$ the set of all $\ell \in \mathbf N^+$ such that $
a_1 \cdots a_k = b_1 \cdots b_{\ell}$ for some atoms $a_1, \ldots, a_k, b_1, \ldots, b_\ell \in H$ (see also Example \ref{exa:system-of-sets-of-lengths}), where an element of $H$ is an atom if it is neither a unit nor the product of two non-units: The sets $\UUc_k(H)$ are called \textit{unions of sets of lengths} and have been studied in factorization theory since decades, see \cite{Fa-Zh16a} for recent progress and \cite{Ge16c, Sc16a} for surveys.
In particular, we say that $H$ satisfies the \textit{Structure Theorem for Unions} if there exist $d \in \mathbf N^+$ and $M \in \mathbf N$ such that, for all but finitely many $k \in \mathbf N$,
$$
(k + d \cdot \mathbf Z) \cap \llb \inf \UUc_k(H) + M, \sup \UUc_k(H) - M \rrb \subseteq \mathscr U_k(H) \subseteq k + d \cdot \mathbf Z.
$$
The Structure Theorem for Unions holds for a wealth of cancellative monoids \cite{Ga-Ge09b, Ge16c}, and recent work  has revealed that the theorem admits a ``purely additive'' counterpart:
This was made possible by the introduction of \textit{directed families}, and has led, for the first time, to the extension of the theorem to a non-cancellative setting, see \cite[Theorem 2.2 and \S{ }3]{FGKT} and \cite[Theorem 3.6]{GeSchw17}.

Along the same lines of thought, the present paper is aimed to establish
a kind of periodicity of directed families
that applies primarily to unions of sets of lengths:
Nothing similar had been known so far, modulo the fact that, for important but rather special categories of monoids and domains,
the sets $\mathscr U_k$ are arithmetic progressions, if not even intervals as in the case of the ring of integers of a number field or, more in general, of a commutative Krull monoid with finite class group such that each class contains a prime, see \cite[Theorem 4.1]{Fr-Ge08}.
Moreover, some of the achievements of this work will probably help with one of the long term goals in all studies on unions of sets of lengths: To prove a realization theorem in the same spirit of what has already been done with sets of lengths \cite{Sc09a} and sets of distances \cite{Ge-Sc17a}.

With these ideas in mind, we state two of the main contributions of the manuscript. We start with:
\begin{theorem}
\label{th:structure-theorem}
Let $H$ be a monoid, and assume there is $K \in \mathbf N$ such that
$
\sup \UUc_{k+1}(H) \le \sup \UUc_k(H) + K < \infty$ and $
\inf \UUc_k(H) - K \le \inf \UUc_{k+1}(H)$ for all large $k \in \mathbf N$.
Then $H$ satisfies the Structure Theorem for Unions.
\end{theorem}
We will use (a purely additive version of) Theorem \ref{th:structure-theorem} to obtain a substantial refinement of the Struc\-ture Theorem for Unions. For, we say that $H$ has \textit{accepted elasticity} if the supremum of the set
$$
\{m/n: a_1 \cdots a_m = b_1 \cdots b_n \text{ for some atoms }a_1, \ldots, a_m, b_1, \ldots, b_n \in H\} \subseteq \mathbf Q^+
$$
is attained or zero.
Further, we denote by $\Delta(H)$ the \textit{set of distances} of $H$, i.e., the set of all $d \in \mathbf N^+$ for which there are $x \in H$ and $k \in \mathbf N^+$ such that $x$ has factorizations (into irreducible elements of $H$) of length $k$ and $k+d$, but $x \ne a_1 \cdots a_\ell$ for every $\ell \in \llb k+1, k+d-1 \rrb$ and all atoms $a_1, \ldots, a_\ell \in H$. Then we have:
\begin{theorem}
\label{th:main-theorem-intro}
Let $H$ be a monoid with accepted elasticity. Then $H$ satisfies the Structure Theorem for Unions and there exists $\mu \in \mathbf N^+$ such that, for every $M \in \mathbf N$,
$$
\bigl((\mathscr U_k - \inf \mathscr U_k) \cap \llb 0, M \rrb\bigr)_{k \ge 1}
\quad\text{and}\quad
\bigl((\sup \mathscr U_k - \mathscr U_k) \cap \llb 0, M \rrb \bigr)_{k \ge 1}
$$
are $\mu$-periodic sequences from some point on.
\end{theorem}
Theorem \ref{th:main-theorem-intro} applies in the first place to
(the multiplicative monoid of) all commutative Krull domains (e.g., Dedekind domains) with finite class group,
to some maximal orders in central simple algebras over global fields, and to a wide class of weakly Krull commutative domains (including all orders in algebraic number fields with finite elasticity); see \S{ }\ref{sec:focus-on-sets-of-lengths} for references and further applications.

As a matter of fact, we will not prove Theorems \ref{th:structure-theorem} and \ref{th:main-theorem-intro} directly: We will rather derive them from more general results on subadditive subfamilies of $\mathcal P(\mathbf N)$, which are the object of \S{ }\ref{sec:basic-properties-of-directed-families} (thus, we postpone the proofs of Theorems \ref{th:structure-theorem} and \ref{th:main-theorem-intro} to \S{ }\ref{sec:focus-on-sets-of-lengths}).
\subsection{Generalities}
\label{sec:notations}
Unless noted otherwise, we reserve the letters $d$, $m$, and $n$ (with or without subscripts) for positive integers, and the letters $h$, $i$, $j$, $k$, and $\kappa$ for non-negative integers. We use $\mathbf R$ for the reals, $\mathbf Q$ for the rationals, $\mathbf Z$ for the integers, and $\mathbf N$ for the non-negative integers.

We let a \textit{monoid} be a pair $(H, \otimes)$ consisting of a set $H$, systematically identified with the monoid itself if there is no danger of confusion, and an associative (binary) operation $\otimes: H \times H \to H$ for which there exists a (provably unique) element $e \in H$, the \textit{identity} of the monoid, such that $e \otimes x = x \otimes e = x$ for all $x \in H$.
We assume that monoid homomorphisms preserve the identity.

If $(H, \otimes\fixed[0.2]{\text{ }})$ is a monoid and $X, Y \subseteq H$, we set $X \fixed[-0.2]{\text{ }} \otimes Y := \{x \otimes y: (x,y) \in X \times Y\}$, and we denote by $H^\times$ the \textit{group of units} (or \textit{invertible elements}) of $H$; accordingly, we write $x \simeq_H y$, for $x, y \in H$, if there exist $u, v \in H^\times$ such that $x = u \otimes y \otimes v$.

If $a,b \in \mathbf R \cup \{\pm \infty\}$ and $d \in \mathbf N^+$, we let $\llb a, b \rrb := \{x \in \mathbf Z: a \le x \le b\}$
stand for the (discrete) interval between $a$ and $b$, and we take an \textit{arithmetic progression} (shortly, AP) \textit{with difference $d$} to be a set of the form $x + d \cdot \llb y, z \rrb$ with $x \in \mathbf Z$  and $y,z \in \mathbf Z \cup \{\pm \infty\}$ (note that an AP need not be finite or non-empty).

If $\lambda \in \mathbf R$ and $X, Y \subseteq \mathbf R$, we denote by $X^+$ the positive part of $X$ (so, $\mathbf N^+$ is the set of positive integers), and we define the sumset of $X$ and $Y$ by $X + Y := \{x+y: (x,y) \in X \times Y\}$, the $n$-fold sumset of $X$ by $
nX := \{x_1 + \cdots + x_n: x_1, \ldots, x_n \in X\}$, and the $\lambda$-dilation of $X$ by
$\lambda \cdot \fixed[-0.3]{\text{ }} X := \{\lambda x: x \in X\}$.

We let $\mathfrak{S}_n$ be the group of permutations of the interval $\llb 1, n \rrb$, and
we write $\mathcal P(X)$ for the power set of a set $X$.
Lastly, we adopt the con\-ven\-tion that $\sup \emptyset = \gcd\emptyset = \infty - \infty = 0 \cdot \infty = \infty \cdot 0 = \frac{a}{\infty} := 0$ and
$\inf \emptyset = \frac{a}{0} := \infty$ for every $a \in {[0,\infty[}\fixed[0.25]{\text{ }}$.

Further no\-ta\-tions and terminology, if not explained, are standard or should be clear from the context.
\section{Subadditive families}
\label{sec:basic-properties-of-directed-families}
In this section, we introduce, and prove several properties of, subadditive families: Some are refinements of analogous properties established in \cite[\S{ }2]{FGKT} under stronger conditions.

To begin, let $\mathscr{L}$ be a collection of (finite or infinite) subsets of $\mathbf N$. Given $i \in \mathbf N^+$ and $ k \in \mathbf N$, we define
$$
\mathscr{U}_k(\mathscr L) := \mathscr{U}_{k,1}(\mathscr L) := \bigcup \fixed[-0.2]{\text{ }} \{L \in \mathscr{L}: k \in L\}
\quad\text{and}\quad
\mathscr{U}_{k,i+1}(\mathscr L) := \mathscr{U}_{k,i}(\mathscr L) \setminus \{\lambda_{k,i}(\mathscr{L}),\rho_{k,i}(\mathscr L)\},
$$
where $
\lambda_{k,i}(\mathscr{L}) := \inf \mathscr{U}_{k,i}(\mathscr{L})
$ and $
\rho_{k,i}(\mathscr{L}) := \sup \mathscr{U}_{k,i}(\mathscr{L})$;
in particular, we take
$$
\lambda_k(\mathscr{L}) := \lambda_{k,1}(\mathscr L)
\quad\text{and}\quad
\rho_k(\mathscr L) := \rho_{k,1}(\mathscr L).
$$
We refer to $\rho_k(\mathscr{L})$ and $\lambda_k(\mathscr{L})$, respectively, as the \textit{$k$-th upper} and the \textit{$k$-th lower} \textit{local elasticity} of $\mathscr{L}$.

We write $\rho(\mathscr{L})$ for the supremum of $\rho(L) := \sup L/\inf L^+$ as $L$ ranges over $\mathscr{L}$, and
we set $\lambda(\mathscr{L}) := 1/\rho(\mathscr{L})$.
We call $\rho(\mathscr{L})$ and $\lambda(\mathscr{L})$, re\-spec\-tive\-ly, the \textit{upper} and the \textit{lower elasticity} of $\mathscr{L}$:
Since we assume $\inf \emptyset := \infty$, it is clear that
$\{\rho(L): L \in \mathscr{L}\} \subseteq \{0\} \cup [1,\infty]$, and hence $\rho(\LLc) = 0$ or $1 \le \rho(\LLc) \le \infty$.
We say that $\mathscr{L}$ has \textit{accepted elasticity} if $\LLc = \emptyset$ or $\rho(\LLc) = \rho(L) < \infty$ for some $L \in \mathscr{L}$.

We take $\wp(\LLc)$ to be the greatest common divisor of the set $\bigcup \fixed[-0.2]{\text{ }} \{L^+: L \in \LLc\} \subseteq \mathbf N^+$. Observe that, in our conventions, $\wp(\LLc)$ is a non-negative integer, with $\wp(\LLc) = 0$ if and only if $\LLc \subseteq \fixed[-0.2]{\text{ }} \bigl\{\emptyset, \{0\}\bigr\}$.

We denote by $\Delta(L)$, for a given $L \subseteq \mathbf N$, the set of all integers $d \ge 1$ such that there exists $\ell \in L$ with $L \cap \llb \ell, \ell + d \fixed[0.2]{\text{ }} \rrb = \{\ell, \ell + d\}$. Accordingly,  we let
$$
\DeltaF{\LLc} := \bigcup_{L \in \mathscr L} \Delta(L)
\quad\text{and}\quad
\Delta_\cup(\LLc) := \bigcup_{k \ge 0} \Delta(\mathscr U_k(\mathscr L)).
$$
We call $\DeltaF{\LLc}$ the \textit{set of distances} (or \textit{delta set}) of $\mathscr L$, and we define $\delta(\LLc) := \inf \Delta(\LLc)$. It is trivial that $\Delta(\LLc) \subseteq \mathbf N^+$ and $\delta(\LLc) \in \mathbf N^+ \cup \{\infty\}$, with $\delta(\LLc) = \infty$ if and only if $\Delta(\LLc) = \emptyset$.

Lastly, we say that $\mathscr{L}$ is: \textit{finitary} if $|L| < \infty$ for all $L \in \mathscr{L}$; \textit{subadditive}
if for all $L_1, L_2 \in \mathscr{L}$ there is a set $L \in \mathscr{L}$ with $L_1 + L_2 \subseteq L$; \textit{directed} if it is subadditive and $1 \in L^\prime$ for some $L^\prime \in \mathscr{L}$; and \textit{primitive} if $\wp(\LLc) = 1$. Note that every directed family is primitive.

We will usually omit the dependence of the above quantities on $\mathscr{L}$ when $\mathscr{L}$ is implied from the context, so as to write $\rho$ in place of $\rho(\mathscr{L})$, $\mathscr{U}_k$ instead of $\mathscr{U}_k(\mathscr{L})$, etc.

The following are key examples of subadditive, directed, or finitary families we shall have in mind: The second of them is of great importance in factorization theory and will be the focus of \S{ }\ref{sec:focus-on-sets-of-lengths}.
\begin{example}
\label{exa:systems-of-sets-of-lengths}
Let $H$ be a multiplicatively written monoid with identity $1_H$; $A$ a subset of $H$ such that $1_H \notin \langle A \rangle_H$, where $\langle A \rangle_H$ is the subsemigroup of $H$ generated by $A$; and $\eta$ a function $A \to \mathbf N$, which, roughly speaking, assigns a (non-negative integral) ``weight'' to each element of $A$.

We set $\LLs_H(1_H; \eta) := \{0\} \subseteq \mathbf N$, and for every $x \in H \setminus \{1_H\}$ we take $\LLs_H(x\fixed[0.22]{\text{ }}; \eta) := \{\eta(a_1) + \cdots + \eta(a_n): x = a_1 \cdots a_n\text{ for some }a_1, \ldots, a_n \in A\}$.
We claim that the family
$$
\mathscr{L}(H\fixed[0.22]{\text{ }}; \eta) := \{\mathsf L_H(x\fixed[0.22]{\text{ }}; \eta): x \in H\} \setminus \{\emptyset\} \subseteq \mathcal P(\mathbf N)
$$
is subadditive. Indeed, pick $x, y \in H$ such that $\mathsf L_H(x\fixed[0.22]{\text{ }}; \eta)$ and $\mathsf L_H(y\fixed[0.22]{\text{ }}; \eta)$ are non-empty:
We aim to prove
$$
\mathsf L_H(x\fixed[0.22]{\text{ }}; \eta) + \mathsf L_H(y\fixed[0.22]{\text{ }}; \eta) \subseteq \mathsf L_H(xy\fixed[0.22]{\text{ }}; \eta).
$$
This is obvious if $x$ or $y$ is $1_H$. Otherwise, it suffices to observe that, if $x = a_1 \cdots a_m$ and $y = b_1 \cdots b_n$ for some $a_1, \ldots, a_m, b_1, \ldots, b_n \in A$, and hence $\sum_{i=1}^m \eta(a_i) \in \mathsf L_H(x\fixed[0.22]{\text{ }}; \eta)$ and $\sum_{i=1}^n \eta(b_i) \in \mathsf L_H(y\fixed[0.22]{\text{ }}; \eta)$, then $xy = a_1 \cdots a_m b_1 \cdots b_n \ne 1_H$ (recall that $1_H \notin \langle A \rangle_H$), so that $\sum_{i=1}^m \eta(a_i) + \sum_{i=1}^n \eta(b_i) \in \mathsf L_H(xy\fixed[0.22]{\text{ }}; \eta)$.

An analogous construction, restricted to the case when $A \subseteq H \setminus H^\times$ and $\eta(a) := 1$ for all $a \in A$, was considered in \cite[Example 2.1]{FGKT}, where it is maintained that $\mathscr{L}(H\fixed[0.22]{\text{ }}; A)$ is a subadditive family, with or without the assumption that $H^\times$ is disjoint from $\langle A \rangle_H$: This claim is actually incorrect (though the issue does not affect the main results of \cite{FGKT}), as we can see from \cite[Lemma 2.2 and Proposition 2.30]{FTr} when $A = H \setminus H^\times$ and $H$ is \textit{not} Dedekind-finite (i.e., there are $x, y \in H$ such that $xy = 1_H \ne yx$).

Besides that, our construction can model many more ``real-life situations''. For instance, fix $n \in \mathbf N^+$, and let $G$ be the additive group of the integers modulo $n$; $G_0$ a subset of $G$; and $H$ the monoid of zero-sum sequences over $G$ with support in $G_0$ (see \cite[Definition 2.5.5]{GeHK06} for notations and terminology). We as\-so\-cia\-te to each $x \in G$ a weight $a_x \in \mathbf N$ (e.g., the smallest non-negative integer in the congruence class $x$). Then, we may take $A$ to be the set of all \textit{minimal} zero-sum sequences over $G$ with support in $G_0$, and for every (non-empty) sequence $\mathfrak s = x_1 \cdots x_k \in A$ define $\eta(\mathfrak s) := a_{x_1} + \cdots + a_{x_k}$.

Incidentally, a construction in the same spirit as ours was studied by Halter-Koch in \cite{HK93}, where it is, however, assumed that $H$ is a cancellative, commutative monoid with trivial group of units; $A$ is a finite set with $H = \{1_H\} \cup \langle A \rangle_H$ (in particular, $H$ is finitely generated); and $\eta$ is a function $A \to \mathbf Z$ (that is, Halter-Koch's construction allows signed integral weights, which is not the case in the present work).
\end{example}
\begin{example}
\label{exa:system-of-sets-of-lengths}
Keeping the notations of Example \ref{exa:systems-of-sets-of-lengths}, let $\mathcal A(H)$ denote the \textit{set of atoms} (or \textit{irreducible elements}) of $H$ and $\eta$ the constant map $\mathcal A(H) \to \mathbf N: a \mapsto 1$.
We define $\mathscr{L}(H) := \mathscr{L}(H\fixed[0.22]{\text{ }}; \eta)$, and we set, for every $x \in H$, $\mathsf L_H(x) := \mathsf L_H(x\fixed[0.22]{\text{ }}; \eta)$.
We refer to $\mathscr{L}(H)$ as the \textit{system of sets of lengths} of $H$.

Clearly, $\mathscr{L}(H)$ is a finitary family if $H$ is a BF-monoid, viz., $1 \le |\mathsf L_H(x)| < \infty$ for every $x \in H \setminus H^\times$. Moreover, we have by \cite[Lemma 2.2(i)]{FTr} that $1_H \notin \langle \mathcal A(H) \rangle_H$.
So, if $\mathcal A(H)$ is non-empty, $\mathscr{L}(H)$ is a directed family by the considerations of Example \ref{exa:systems-of-sets-of-lengths} and the fact that $1 \in \LLs_H(a)$ for all $a \in \mathcal A(H)$; otherwise, $\mathscr L(H)$ is equal to $\bigl\{\{0\}\bigr\}$, which is a subadditive family in a trivial way.
\end{example}
\begin{example}
Let $\mathscr L$ be a subadditive family, and fix $\alpha \in \mathbf R$. We want to show that the family
$$
\mathscr L_\alpha := \{L \in \mathscr L: \rho(L) \ge \alpha\} \subseteq \mathscr L
$$ is also subadditive (note that $\mathscr L_\alpha$ need not be directed, no matter whether $\mathscr L$ is).

It is enough to consider the case when $\mathscr L_\alpha$ is non-empty (otherwise the claim is trivial) and $\alpha \in \mathbf R^+$ (otherwise $\mathscr L_\alpha = \mathscr L$, and there is nothing to prove). Accordingly, pick $L_1, L_2 \in \mathscr L_\alpha$.
Since $\mathscr L$ is a sub\-ad\-di\-tive family and $L_1, L_2 \in \mathscr L$, there exists $L \in \mathscr L$ with $L_1 + L_2 \subseteq L$. Also, $L_1^+$ and $L_2^+$ are non-empty, because $\rho(L_1)$ and $\rho(L_2)$ are both positive. It follows that
$$
\rho(L) = \frac{\sup L}{\inf L^+} \ge \frac{\sup (L_1 + L_2)}{\inf (L_1 + L_2)^+} \ge \frac{\sup L_1 + \sup L_2}{\inf L_1^+ + \inf L_2^+} \ge \min (\rho(L_1), \rho(L_2) ) \ge \alpha,
$$
which yields $L \in \mathscr L_\alpha$ and shows that $\mathscr L_\alpha$ is a subadditive family, as wished.
\end{example}
\begin{example}
\label{exa:power-monoids}
Following \cite[\S\S{ }3--4]{FTr}, let $\mathcal P_{\rm fin}(\mathbf N)$ denote the power monoid of $(\mathbf N, +)$, i.e., the set of all non-empty finite subsets of $\mathbf N$ endowed with the operation of \textit{set addition}
$$
\PPc_{\rm fin}(\mathbf N) \times \PPc_{\rm fin}(\mathbf N) \to \PPc_{\rm fin}(\mathbf N): (X, Y) \mapsto X+Y.
$$
Every subsemigroup of $\mathcal P_{\rm fin}(\mathbf N)$ is a finitary, subadditive family, but of course need not be directed.
\end{example}
We proceed to prove a basic result (on the set of distances of a subadditive family) that is essentially an extension of \cite[Proposition 2.9]{FGKT}, where the scope was restricted to directed families.
\begin{proposition}
\label{prop:generalized-deltas}
Let $\LLc \subseteq \mathcal P(\mathbf N)$ be a subadditive family with $\DeltaF{\LLc} \ne \emptyset$, and let $\Delta'$ be a non-empty subset of $\DeltaF{\LLc}$ such that $\gcd \Delta' \le \delta$. Then $\gcd \Delta' = \delta$. In particular, $\delta = \gcd \DeltaF{\LLc}$.

\end{proposition}
\begin{proof}
Set $\delta^{\fixed[0.2]{\text{ }} \prime} := \gcd \Delta'$. Since $\Delta' \ne \emptyset$, we get from \cite[Theorem 1.4]{Nath} that there are $\eps_1, \ldots, \eps_n \in \{\pm 1\} \subseteq \mathbf Z$, $d_1, \ldots, d_n \in \Delta'$, and $m_1, \ldots, m_n \in \mathbf N^+$ such that $\delta^{\fixed[0.2]{\text{ }} \prime} = \eps_1 m_1 d_1 + \cdots + \eps_n m_n d_n$.

In addition, for each $i \in \llb 1, n \rrb$ we can find $x_i \in \mathbf N$ and $L_i \in \LLc$ with $\{x_i, x_i + \eps_i d_i\} \subseteq L_i$. Because $\LLc$ is a subadditive family, this yields
$\{m_i x_i, m_i (x_i + \eps_i d_i)\} \subseteq
m_i L_i \subseteq L_i'$ for some $L_i' \in \LLc$. Moreover, there is a set $L \in \LLc$ such that $L_1' + \cdots + L_n' \subseteq L$. Put $\ell := m_1 x_1 + \cdots + m_n x_n$.

Then we have by the above that
$
\ell + \delta^{\fixed[0.2]{\text{ }} \prime} = \sum_{i=1}^n m_i (x_i + \eps_i d_i)
$,
and we infer that $\ell$ and $\ell+\delta^{\fixed[0.2]{\text{ }} \prime}$ are both in $L$. Thus $\delta \le \inf \Delta(L) \le \delta^{\fixed[0.2]{\text{ }} \prime} = \gcd\Delta'$, which is enough to conclude $\gcd\Delta' = \delta$, in that we are assuming $\gcd\Delta' \le \delta$. (Since $\gcd \DeltaF{\LLc} \le \delta$, the rest is clear.)
\end{proof}
\begin{corollary}
\label{cor:delta_sets}
Let $\LLc \subseteq \mathcal P(\mathbf N)$ be a subadditive family with $\DeltaF{\LLc} \ne \emptyset$. The following hold:
\begin{enumerate}[label={\rm (\roman{*})}]
\item\label{it:cor:delta_sets(i)} If $L \in \LLc$ and $x, y \in L$, then $\delta \mid y - x$.
\item\label{it:cor:delta_sets(ii)} If $x,y \in \UUc_k$ for some $k \in \NNb$, then $\delta \mid y-x$.
\item\label{it:cor:delta_sets(iii)} For every $q \in \mathbf N$, there exist $\ell \in \mathbf N^+$ and $L \in \LLc$ such that $\ell + \delta \cdot \llb 0, q \rrb \subseteq L \subseteq \UUc_{\ell}$.
\end{enumerate}
\end{corollary}
\begin{proof}
\ref{it:cor:delta_sets(i)} Let $L \in \LLc \setminus \{\emptyset\}$ and $x, y \in L$ (observe that $\LLc \not\subseteq \{\emptyset\}$, because, by hypothesis, $\Delta(\LLc) \ne \emptyset$). If $x=y$, the claim is obvious. Otherwise, there are $x_1, \ldots, x_n \in \mathbf N$ such that $x = x_1 < \cdots < x_n = y$ and $L \cap \llb x, y \rrb = \{x_1, \ldots, x_n\}$, where without loss of generality we assume $x < y$. It follows $x_{i+1} - x_i \in \Delta(L)$ for each $i \in \llb 1, n-1 \rrb$ (note that $n \ge 2$), which implies by Proposition \ref{prop:generalized-deltas} that $\delta \mid x_{i+1} - x_i$. So $\delta \mid y-x$, since $y - x = (x_n - x_{n-1}) + \cdots + (x_2 - x_1)$.

\ref{it:cor:delta_sets(ii)} Let $k \in \mathbf N$ such that $\UUc_k \ne \emptyset$, and pick $x, y \in \UUc_k$. Then there exist $L_x, L_y \in \LLc$ with $\{k, x\} \subseteq L_x$ and $\{k, y\} \subseteq L_y$, and we obtain from \ref{it:cor:delta_sets(i)} that $\delta \mid x - k$ and $\delta \mid y - k$. This yields $\delta \mid y - x$.

\ref{it:cor:delta_sets(iii)} Pick $q \in \mathbf N$. Since $\delta \in \Delta(\LLc) \ne \emptyset$, there are $\ell^{\fixed[0.2]{\text{ }}\prime} \in \mathbf N$ and $L^\prime \in \mathscr L$ such that $\{\ell^{\fixed[0.2]{\text{ }}\prime}, \ell^{\fixed[0.2]{\text{ }}\prime} + \delta\} \subseteq L^\prime$. Using that $\LLc$ is a subadditive family, we obtain
$$
(q+1)\ell^{\fixed[0.2]{\text{ }}\prime} + \delta \cdot \llb 0, q+1 \rrb = (q+1) \fixed[0.2]{\text{ }} \{\ell^{\fixed[0.2]{\text{ }}\prime}, \ell^{\fixed[0.2]{\text{ }}\prime} + \delta\} \subseteq (q+1)L^\prime \subseteq L,
$$
for some $L \in \LLc$. So $\ell + \delta \cdot \llb 0, q \rrb \subseteq L \subseteq \UUc_{\ell}$, where $\ell := (q+1)\ell^{\fixed[0.2]{\text{ }}\prime} + \delta \in \mathbf N^+$.
\end{proof}
We continue with a couple of lemmas, the first of which is essentially a revision of \cite[Lemma 2.4]{FGKT}.
\begin{lemma}
\label{prop:basic(1)}
Let $\mathscr{L} \subseteq \mathcal P(\mathbf N)$ be a subadditive family. The following hold:
\begin{enumerate}[label={\rm (\roman{*})}]
\item\label{it:prop:basic(1)(i)} Given $h, k \in \mathbf N$, we have $h \in \UUc_k$ if and only if $k \in \UUc_h$.
\item\label{it:prop:basic(1)(ii)} $\DeltaF{\LLc} = \emptyset$ if and only if $\mathscr U_k \subseteq \{k\}$ for all $k \in \mathbf N$.
\item\label{it:prop:basic(1)(iii)} $\rho_k = \infty$, for some $k \in \mathbf N$, if and only if $\rho_{k,i} = \infty$ for all $i \in \mathbf N^+$.
\item\label{it:prop:basic(1)(iv)} $\UUc_{h,i} + \UUc_{k,j} \subseteq \UUc_{h+k,i+j-1}$ for all $h, k \in \mathbf N$ and $i, j \in \mathbf N^+$.
\item\label{it:prop:basic(1)(v)}
$\lambda_{h+k,i+j-1} \le \lambda_{h,i} + \lambda_{k,j} \le \rho_{h,i} + \rho_{k,j} \le \rho_{h+k,i+j-1}$ for all $h,k \in \mathbf N$ and $i,j \in \mathbf N^+$ such that $\UUc_{h,i}$ and $\UUc_{k,j}$ are non-empty.
\end{enumerate}
\end{lemma}
\begin{proof}
\ref{it:prop:basic(1)(i)} If $h \in \mathscr U_k$, then $h \in L$ for some $L \in \mathscr L$ with $k \in L$, so $k \in \mathscr U_h$ and we are done (by symmetry).

\ref{it:prop:basic(1)(ii)}
$\DeltaF{\LLc} \ne \emptyset$ if and only if $\Delta(L) \ne \emptyset$ for some $L \in \LLc$, i.e., if and only if there exist $\ell \in \mathbf N$, $d \in \NNb^+$ and $L \in \LLc$ with $\{\ell, \ell+d\} \subseteq L$, which implies $\{\ell, \ell + d\} \subseteq \UUc_\ell \not\subseteq \{\ell\}$. Conversely, if $\UUc_k \subsetneq \{k\}$ for some $k$, then there are $h \in \mathbf N$ and $L \in \mathscr L$ with $h \ne k$ and $\{h,k\} \subseteq L$, whence $\emptyset \ne \Delta(L) \subseteq \Delta(\LLc)$.

\ref{it:prop:basic(1)(iii)} The ``if'' part is obvious, so let $k \in \mathbf N$ such that $\rho_k = \infty$. Then $\UUc_k$ is an infinite subset of $\mathbf N$, and, hence, so are $\mathscr U_{k,1}, \UUc_{k,2}, \ldots \fixed[-0.2]{\text{ }}$, because $\UUc_{k,i+1} =  \UUc_k \setminus \{\lambda_{k,1}, \rho_{k,1}, \ldots, \lambda_{k,i}, \rho_{k,i}\} = \UUc_k \setminus \{\lambda_{k,1}, \ldots, \lambda_{k,i}\}$ for all $i \in \mathbf N^+$. Therefore, it is clear that $\rho_{k,1} = \rho_{k,2} = \cdots = \infty$.

\ref{it:prop:basic(1)(iv)} Fix $h,k \in \mathbf N$. First, we prove $\UUc_h + \UUc_k \subseteq \UUc_{h+k}$. This is trivial if $\UUc_h$ or $\UUc_k$ is empty. Otherwise, let $r \in \UUc_h$ and $s \in \UUc_k$. Then $\{r,h\} \subseteq L_1$ and $\{s,k\} \subseteq L_2$ for some $L_1, L_2 \in \LLc$, and since $\LLc$ is subadditive, there is $L \in \LLc$ with $\{r+s, h+k\} \subseteq L_1 + L_2 \subseteq L$. So $r+s \in \UUc_{h+k}$, viz., $\mathscr U_h + \mathscr U_k \subseteq \mathscr U_{h+k}$.

Now, pick $i,j \in \mathbf N^+$. We have to show $\UUc_{h,i} + \UUc_{k,j} \subseteq \UUc_{h+k,i+j-1}$. If $\UUc_{h,i}$ or $\UUc_{k,j}$ is empty, we are done. Otherwise, let $r \in \mathscr U_{h,i}$ and $s \in \mathscr U_{k,j}$: It is sufficient to check that $r+s \in \mathscr U_{h+k,i+j-1}$.
To this end, we infer from the definition of $\UUc_{h,i}$ and $\UUc_{k,j}$ that
\begin{equation*}
\label{equ:chains-of-inequalities}
\underbrace{\lambda_{h,1} \le \cdots \le \lambda_{h,i}}_{({\rm a})} \le r \le \underbrace{\rho_{h,i} \le \cdots \le \rho_{h,1}}_{({\rm b})}
\quad\text{and}\quad
\underbrace{\lambda_{k,1} \le \cdots \le \lambda_{k,j}}_{({\rm c})} \le s \le \underbrace{\rho_{k,j} \le \cdots \le \rho_{k,1}}_{({\rm d})},
\end{equation*}
where, as implied by \ref{it:prop:basic(1)(iii)}, the inequalities labeled by (a) (respectively, by (c)) are strict if and only if $i \ge 2$ (respectively, $j \ge 2$), and the inequalities labeled by (b) (respectively, by (d)) are strict if and only if $i \ge 2$ and $\rho_h < \infty$ (respectively, $j \ge 2$ and $\rho_k < \infty$). So, it is straightforward that, on the one hand,
\begin{equation}
\label{equ:adding-up-a-chain-of-local-lower-elasticities-lower}
\rlap{$\overbrace{\phantom{\lambda_{h,1} + \lambda_{k,1} \le \cdots \le \lambda_{h,i} + \lambda_{k,1}}}^{({\rm A})}$} \lambda_{h,1} + \lambda_{k,1} \le \cdots \le \underbrace{\lambda_{h,i} + \lambda_{k,1} \le \cdots \le \lambda_{h,i} + \lambda_{k,j}}_{({\rm B})} \le r+s,
\end{equation}
and on the other hand,
\begin{equation}
\label{equ:adding-up-a-chain-of-local-upper-elasticities-upper}
r+s \le \rlap{$\overbrace{\phantom{\rho_{h,i} + \rho_{k,j} \le \cdots \le \rho_{h,1} + \rho_{k,j}}}^{({\rm C})}$}
\rho_{h,i} + \rho_{k,j} \le \cdots \le \underbrace{\rho_{h,1} + \rho_{k,j} \le \cdots \le \rho_{h,1} + \rho_{k,1}}_{({\rm D})},
\end{equation}
with the inequalities labeled by (A) (respectively, by (B)) being strict if and only if $i \ge 2$ (respectively, $j \ge 2$), and the inequalities labeled by (C) (respectively, by (D)) being strict if and only if $i \ge 2$ (respectively, $j \ge 2$) and $\max(\rho_h, \rho_k) < \infty$. Moreover, if $\rho_h < \infty$ then $\lambda_{h,1}, \rho_{h,1}, \ldots, \lambda_{h,i},\rho_{h,i} \in \UUc_h$; and in a similar way, if $\rho_k < \infty$ then $\lambda_{k,1}, \rho_{k,1}, \ldots, \lambda_{k,j},\rho_{k,j} \in \UUc_k$.

So, putting it all together and using that $\UUc_h + \UUc_k \subseteq \UUc_{h+k}$, we infer from \eqref{equ:adding-up-a-chain-of-local-lower-elasticities-lower}, \eqref{equ:adding-up-a-chain-of-local-upper-elasticities-upper}, and \ref{it:prop:basic(1)(iii)} that
$$
r+s \in \UUc_{h+k},
\quad
\bigl|\UUc_{h+k} \cap \llb 0, r+s \rrb\bigr| \ge i+j-1,
\quad\text{and}\quad
\bigl|\UUc_{h+k} \cap \llb r+s, \infty \rrb\bigr| \ge i+j-1,
$$
which yields $r+s \in \UUc_{h+k,i+j-1}$, and hence $\UUc_{h,i} + \UUc_{k,j} \subseteq \UUc_{h+k,i+j-1}$.

\ref{it:prop:basic(1)(v)} Assume that $\UUc_{h,i}$ and $\UUc_{k,j}$ are non-empty for some $h, k \in \mathbf N$ and $i,j \in \mathbf N^+$. Then it is obvious that $\lambda_{h,i} \le \rho_{h,i}$ and $\lambda_{k,j} \le \rho_{k,j}$. On the other hand, we obtain from \ref{it:prop:basic(1)(iv)} that $\UUc_{h,i} + \UUc_{k,j} \subseteq \UUc_{h+k,i+j-1} \ne \emptyset$, and this implies $\lambda_{h+k,i+j-1} \le \lambda_{h,i} + \lambda_{k,j}$ and $\rho_{h,i} + \rho_{k,j} \le \rho_{h+k,i+j-1}$.
\end{proof}
\begin{lemma}
\label{lem:non-emptyness-of-multiples}
Suppose $\mathscr{L} \subseteq \mathcal P(\mathbf N)$ is a subadditive family. The following hold:
\begin{enumerate}[label={\rm (\roman{*})}]
\item\label{it:lem:non-emptyness-of-multiples(i)} $\UUc_k = \emptyset$ for every $k \notin \wp \cdot \NNb^+$.
\item\label{it:lem:non-emptyness-of-multiples(ii)} If $\wp \ne 0$, then there exists $k_0 \in \mathbf N$ such that $\UUc_{\wp k} \ne \emptyset$ for all $k \ge k_0$.
\item\label{it:lem:non-emptyness-of-multiples(iiiq)}If $\rho_\kappa = \infty$ for some $\kappa \in \mathbf N$, then $\rho_{\wp k} = \infty$ for all but finitely many $k$.
\item\label{it:lem:non-emptyness-of-multiples(iii)} $\wp \mid \gcd \DeltaF{\LLc}$.
\item\label{it:lem:non-emptyness-of-multiples(iv)} Pick $i \in \mathbf N^+$, and assume $\DeltaF{\LLc}$ is non-empty. Then $\UUc_{\wp k,i} \ne \emptyset$, and hence $\lambda_{\wp k,1} \le \cdots \le \lambda_{\wp k,i} \le \wp k \le \rho_{\wp k,i} \le \cdots \le \rho_{\wp k,1}$, for all large $k \in \mathbf N$.
\end{enumerate}
\end{lemma}
\begin{proof}
\ref{it:lem:non-emptyness-of-multiples(i)} If $\wp = 0$, then $L^+ = \emptyset$ for every $L \in \LLc$, and hence $\UUc_k = \emptyset$ for all $k \in \NNb^+$; so we are done, because $0 \mid k$, for some $k \in \NNb$, if and only if $k = 0$.
If, on the other hand, $\wp \ge 1$ and $\UUc_k \ne \emptyset$ for some $k \in \mathbf N^+$, then it is clear from our definitions that $\wp \mid \gcd(\UUc_k^+)$, and hence $\wp \mid k$.

\ref{it:lem:non-emptyness-of-multiples(ii)} We have by \cite[Theorem 1.4]{Nath} that $\wp = \varepsilon_1 m_1 k_1 + \cdots + \varepsilon_n m_n k_n$ for some $\varepsilon_1, \ldots, \varepsilon_n \in \{\pm 1\} \subseteq \mathbf Z$, $m_1, \ldots, m_n \in \NNb^+$, and $k_1, \ldots, k_n \in \bigcup \fixed[-0.2]{\text{ }} \{L^+: L \in \LLc\}$. Accordingly, put
$$
\ell := 1 + \frac{2}{\wp} (m_1 k_1 + \cdots + m_n k_n).
$$
Then $\ell \in \NNb^+$, since $\wp \mid k_i$ for every $i \in \llb 1, n \rrb$, and we find that
\begin{equation}
\label{equ:a-step-towards-frobenius}
\wp \fixed[0.2]{\text{ }} \ell = a_1 k_1 + \cdots + a_n k_n
\quad\text{and}\quad
\wp (\ell + 1) = b_1 k_1 + \cdots + b_n k_n,
\end{equation}
where $a_i := (2 + \varepsilon_i) m_i \in \mathbf N^+$ and $b_i := (2 + 2\varepsilon_i) m_i \in \NNb$ for $i \in \llb 1, n \rrb$.
Let $k \ge (\ell - 1) \ell + 1$.

By \cite[Theorem 1.7]{Nath}, there are $x, y \in \NNb$ with $x+y \ge 1$ such that $\ell x + (\ell + 1) y = k$. So, we derive from \eqref{equ:a-step-towards-frobenius} that $\wp k = \alpha_1 k_1 + \cdots + \alpha_n k_n \ge 1$, where $\alpha_i := a_i x + b_i y \in \mathbf N$ for each $i \in \llb 1, n \rrb$.

On the other hand, for every $i \in \llb 1, n \rrb$ there exists $L_i \in \LLc$ with $k_i \in L_i$, and since $\LLc$ is a subadditive family and at least one of $\alpha_1, \ldots, \alpha_n$ is positive, it follows $k \in \alpha_1 L_1 + \cdots + \alpha_n L_n \subseteq L$ for some $L \in \LLc$. This yields $\emptyset \ne L \subseteq \UUc_{\wp k}$ and proves the assertion of the lemma with $k_0 = (\ell - 1) \ell + 1$.

\ref{it:lem:non-emptyness-of-multiples(iiiq)} Assume that $\rho_\kappa = \infty$ for some $\kappa \in \mathbf N$. Then $\mathscr U_\kappa^+ \ne \emptyset$ and $0 \ne \wp \mid \kappa$, and it follows from \ref{it:lem:non-emptyness-of-multiples(ii)} that there exists $k_0 \in \mathbf N$ such that $\mathscr U_{\wp k} \ne \emptyset$ for all $k \ge k_0$. So, by Lemma \ref{prop:basic(1)}\ref{it:prop:basic(1)(v)}, $\rho_{\wp k} \ge \rho_{\wp k - \kappa} + \rho_\kappa \ge \infty$ for every $k \ge k_0 + \kappa$ (note that $\wp k - n$ is a multiple of $\wp$).

\ref{it:lem:non-emptyness-of-multiples(iii)} If $\DeltaF{\LLc} = \emptyset$, then $\gcd \Delta(\LLc) = 0$ (by our conventions), and the conclusion is trivial. Otherwise, we have by Proposition \ref{prop:generalized-deltas} that $\gcd \Delta(\LLc) = \delta \ge 1$, so there are $\ell \in \mathbf N$ and $L \in \LLc$ with $\{\ell, \ell + \delta\} \subseteq L$. Using that $\wp$ is the greatest common divisor of $\bigcup \fixed[-0.2]{\text{ }} \{L^+: L \in \LLc\}$, it follows that $\wp \mid \gcd(\ell, \ell + \delta)$, because $\ell + \delta \in L^+$ and, in addition, $\ell \in L^+$ unless $\ell = 0$. Thus, we have $\wp \mid \delta$.

\ref{it:lem:non-emptyness-of-multiples(iv)} Since $\DeltaF{\LLc}$ is non-empty and $\LLc$ is a subadditive family,
we obtain from Corollary \ref{cor:delta_sets}\ref{it:cor:delta_sets(iii)} that there exist $\ell \in \mathbf N^+$ and $L \in \LLc$ with $
\ell + \delta \cdot \llb 0, 2i \rrb \subseteq L$.

In particular, $\wp$ is a positive integer, and we get from \ref{it:lem:non-emptyness-of-multiples(ii)} that there is $k_0 \in \mathbf N$ such that $\mathscr U_{\wp k} \ne \emptyset$ for $k \ge k_0$. Moreover, we have by \ref{it:lem:non-emptyness-of-multiples(iii)} that $\wp \mid \delta$, while it is clear that $\wp \mid \ell$.

Accordingly, fix $k \ge k_0 + (\ell + i\delta)/\wp$. Then $k - (\ell + i\delta)/\wp$ is an integer $\ge k_0$,
and we infer from the above that $\UUc_{\wp k - (\ell + i\delta)} \ne \emptyset$ and $\ell + \delta \cdot \llb 0, 2i \rrb \subseteq L \subseteq \UUc_{\ell + i\delta}$, which, together with Lemma \ref{prop:basic(1)}\ref{it:prop:basic(1)(iv)}, implies
$$
\UUc_{\wp k} \supseteq \mathscr U_{\wp k - (\ell+i\delta)} + \UUc_{\ell + i\delta} \supseteq \wp k - (\ell + i\delta) + \fixed[-0.2]{\text{ }} \bigl(\ell + \delta \cdot \llb 0, 2i \rrb\bigr) \fixed[-0.2]{\text{ }} = \wp k + \delta \cdot \llb -i, i \rrb.
$$
This proves $\bigl|\UUc_{\wp k} \cap \llb 0, \wp k-1 \rrb\bigr| \ge i$ and $\bigl|\UUc_{\wp k} \cap \llb \wp k+1, \infty \rrb\bigr| \ge i$, whence we conclude $\UUc_{\wp k,i} \ne \emptyset$.
\end{proof}
\begin{remark}
\label{rem:primitiveness}
Let $\LLc \subseteq \mathcal P(\mathbf N)$ be a subadditive family. If $\wp(\mathscr L) = 0$, then we have already observed that $\LLc \subseteq \fixed[-0.2]{\text{ }} \bigl\{\emptyset, \{0\}\bigr\}$, and hence $\UUc_k(\mathscr L) = \emptyset$ for all $k \ge 1$. Otherwise, $\wp(\mathscr L)$ is a positive integer and
$$
\mathscr L^\ast := \{\wp(\mathscr L)^{-1} \cdot L: L \in \mathscr L\} \subseteq \mathcal P(\mathbf N)
$$
is also a subadditive family, but with $\wp(\LLc^\ast) = 1$. Since $\UUc_{\wp k}(\LLc) = \wp(\LLc) \cdot \UUc_k(\LLc^\ast)$ for all $k \in \mathbf N$ and, by Lemma \ref{lem:non-emptyness-of-multiples}\ref{it:lem:non-emptyness-of-multiples(i)}, $\UUc_k(\LLc) = \emptyset$ for every $k \notin \wp(\LLc) \cdot \mathbf N$, it follows that, when it comes to structural properties of unions for subadditive families, we can restrict our attention to the ``primitive case'', which is what we will usually do in the remainder of the section.
\end{remark}
The next step is to generalize \cite[Propositions 2.7 and 2.8]{FGKT} from directed to subadditive families:
In fact, our generalization of \cite[Proposition 2.7]{FGKT} is partial, but still sufficient for the goals of the paper.
\begin{lemma}
\label{lem:inequalities_with_elasticities}
Let $\LLc \subseteq \mathcal P(\mathbf N)$ be a subadditive family. The following hold:
\begin{enumerate}[label={\rm (\roman{*})}]
\item\label{it:lem:inequalities_with_elasticities(i)} If $\rho = 0$, then $\mathscr L \subseteq \fixed[-0.15]{\text{ }}\bigl\{\emptyset,\{0\}\bigr\}$, and hence $\rho_k = \rho = 0$ and $\lambda_k = \lambda = \infty$ for all $k \in \mathbf N^+$.
\item\label{it:lem:inequalities_with_elasticities(ii)}  If $\rho < \infty$, then there does not exist any set $L \in \mathscr L$ with $0 \in L$ and $|L| \ge 2$.
\item\label{it:lem:inequalities_with_elasticities(iii)} $k \rho \ge \rho_k$ and $k \lambda \le \lambda_k$ for all $k \in \mathbf N^+$.
\item\label{it:lem:inequalities_with_elasticities(iv)} Assume that $0 < \rho < \infty$ and there are $n \in \mathbf N^+$ and $L \in \mathscr L$ such that $n \rho \le \sup L$ and $\inf L \le n$. Then $\sup L = n \rho$, $\inf L = n$, and $\rho(L) = \rho$ \textup{(}i.e., $\mathscr L$ has accepted elasticity\textup{)}.
\end{enumerate}
\end{lemma}
\begin{proof}
\ref{it:lem:inequalities_with_elasticities(i)} This is trivial by our definitions (in particular, recall that $\lambda := 1/\rho$ and $1/0 := \infty$).

\ref{it:lem:inequalities_with_elasticities(ii)} Suppose to the contrary that there exists $L \in \mathscr L$ with $0 \in L$ and $|L| \ge 2$, and set $\ell := \inf L^+$. Since $L^+ \ne \emptyset$, $\ell$ is an integer $\ge 1$, and it follows from $\mathscr L$ being subadditive that, for each $k \in \mathbf N^+$, there is $L_k \in \mathscr L$ with $\ell \cdot \llb 0,k \rrb = k \fixed[0.2]{\text{ }} \{0, \ell\} \subseteq kL \subseteq L_k$, with the result that $\rho \ge \sup L_k/\inf L_k^+ \ge k$, and hence $\rho = \infty$ (a contradiction).

\ref{it:lem:inequalities_with_elasticities(iii)} The claim is obvious if $\rho = \infty$ (or equivalently, $\lambda = 0$), and it is trivial for every $k \in \mathbf N^+$ for which $\UUc_k = \emptyset$, because this implies, according to our conventions, that $\rho_k = 0$ and $\lambda_k = \infty$.
So we can assume from here on that $\rho < \infty$ and restrict attention to the indices $k \in \mathbf N^+$ such that $\UUc_k \ne \emptyset$.

Based on these premises, we first prove the claim for the upper elasticities, and then we use it for the ``dual statement'' about the lower elasticities:

\textsc{Part 1:} Let $k \in \mathbf N^+$ such that $\UUc_k \ne \emptyset$. Then $\mathscr L_k := \{L \in \mathscr L: k \in L\}$ is a non-empty subfamily of $\mathcal P(\mathbf N)$, and we get from \ref{it:lem:inequalities_with_elasticities(ii)} that, for every $L \in \mathscr L_k$, $\inf L$ is a \textit{positive} integer $\le k$. To wit,
\begin{equation}
\label{equ:bounds_of_the_ration_sup_L/k}
\frac{\sup L}{k} \le \frac{\sup L}{\inf L} = \rho(L) \le \rho < \infty, \ \text{for all }L \in \mathscr L_k.
\end{equation}
In particular, $\rho_k = \sup L^\ast < \infty$ for some $L^\ast \in \mathscr L_k$, which, together with \eqref{equ:bounds_of_the_ration_sup_L/k}, yields $\rho_k \le k \rho$.

\textsc{Part 2:} Again, let $k \in \mathbf N^+$ such that $\UUc_k \ne \emptyset$.
By \ref{it:lem:inequalities_with_elasticities(ii)} and Lemma \ref{prop:basic(1)}\ref{it:prop:basic(1)(i)}, $\lambda_k \in \mathbf N^+$ and $k \in \UUc_{\lambda_k} \ne \emptyset$. So, it follows by the previous part that $\lambda_k \rho \ge \rho_{\lambda_k} \ge k \ge 1$, i.e., $k\lambda \le \lambda_k$.

\ref{it:lem:inequalities_with_elasticities(iv)} Suppose to the contrary that $n \rho < \sup L$ or $\inf L < n$. Since $\rho < \infty$ and $0 < n\rho \le \sup L$, we get from \ref{it:lem:inequalities_with_elasticities(ii)} that $L = L^+ \ne \emptyset$. Hence $\rho(L) = \sup L/\inf L > n \rho/n = \rho$, which is, however, impossible.
\end{proof}
Incidentally, Lemma \ref{lem:inequalities_with_elasticities}\ref{it:lem:inequalities_with_elasticities(ii)} refines \cite[Lemma 2.13(1)]{FGKT} and simplifies the proof of \cite[Theorem 2.2(2)]{FGKT}.
\begin{lemma}
\label{lem:accepted_elasticity}
Assume that $\LLc \subseteq \mathcal P(\mathbf N)$ is a subadditive family with accepted non-zero elasticity, let $L \in \mathscr L$ such that $\rho = \rho(L)$, and set $n := \inf L$. The following hold:
\begin{enumerate}[label={\rm (\roman{*})}]
\item\label{it:lem:accepted_elasticity(i)} If $k \in \mathbf N^+$ and $kL \subseteq L^\prime$ for some $L^\prime \in \LLc$, then $\sup L^\prime = nk \rho$, $\inf L^\prime = nk$, and $\rho(L^\prime) = \rho$.
\item\label{it:lem:accepted_elasticity(ii)}
    $nk \rho = \rho_{nk}$ and $nk = \lambda_{nk\rho}$ for all $k \in \mathbf N^+$ \textup{(}note that $n \rho$ is a non-negative integer\textup{)}.
\end{enumerate}
\end{lemma}
\begin{proof}
\ref{it:lem:accepted_elasticity(i)} Let $k \in \mathbf N^+$ and $L^\prime \in \mathscr L$ such that $kL \subseteq L^\prime$.
Since $\mathscr L$ has accepted non-zero elasticity, we have $1 \le \rho < \infty$, and Lemma \ref{lem:inequalities_with_elasticities}\ref{it:lem:inequalities_with_elasticities(ii)} gives that $L$ and $L^\prime$ are (non-empty) finite subsets of $\mathbf N^+$. Accordingly, we conclude from $kL \subseteq L^\prime$ that $1 \le \inf L^\prime \le nk$ and $k \sup L \le \sup L^\prime < \infty$. It follows
$$
\frac{\sup L^\prime}{\inf L^\prime} = \rho(L^\prime) \le \rho = \rho(L) = \frac{\sup L}{n} = \frac{k \sup L}{nk} \le \frac{\sup L^\prime}{\inf L^\prime},
$$
where the right-most inequality is strict unless $\inf L^\prime = nk$ and $\sup L^\prime = k \sup L$, and it cannot be strict, otherwise we would have a contradiction. This finishes the proof, as it shows that $\rho(L^\prime) = \rho(L) = \rho$.

\ref{it:lem:accepted_elasticity(ii)}
Pick $k \in \mathbf N^+$. Because $\LLc$ is a subadditive family and $L^+ = L$ (as we have already noted), we have $nk \in \NNb^+$ and $kL \subseteq L_k$ for some $L_k \in \mathscr L$. Hence,
$$
\frac{\rho_{nk}}{nk} \ge \frac{\sup L_k}{nk}
\fixed[-0.25]{\text{ }}\stackrel{\ref{it:lem:accepted_elasticity(i)}}{=}\fixed[-0.25]{\text{ }}
\rho \ge \frac{\rho_{nk}}{nk},
$$
where the last inequality is derived from Lemma \ref{lem:inequalities_with_elasticities}\ref{it:lem:inequalities_with_elasticities(iii)}. So, we see that $\rho_{nk} = nk \rho$.

On the other hand, it is clear from the above that $n \rho \in \mathbf N^+$ and $\rho_{nk} < \infty$. In particular, we find that $\{nk, \rho_{nk}\} \subseteq L^\prime$ for some $L^\prime \in \mathscr L$, which, in turn, yields $\sup L^\prime = \rho_{nk} = nk \rho$. Consequently, we obtain from Lemma \ref{lem:inequalities_with_elasticities}\ref{it:lem:inequalities_with_elasticities(iv)} that $\inf L^\prime = nk$ (recall that $nk \in \mathbf N^+$), and since $\rho \lambda = 1$, we conclude
$$
\lambda_{nk\rho} = \inf \UUc_{nk\rho} \le \inf L^\prime = nk = nk \rho \lambda \le \lambda_{nk \rho},
$$
where again, for the last inequality, we use Lemma \ref{lem:inequalities_with_elasticities}\ref{it:lem:inequalities_with_elasticities(iii)}. So $\lambda_{nk \rho} = nk$, and we are done.
\end{proof}
As a side remark, Lemma \ref{lem:accepted_elasticity}\ref{it:lem:accepted_elasticity(ii)} fixes a mistake in the proof of an analogous (and less general) claim used as an intermediate step in the proof of \cite[Theorem 2.2(2)]{FGKT}.
\begin{proposition}
\label{prop:equivalent_conditions_for_accepted_elasticity}
Let $\mathscr L \subseteq \mathcal P(\mathbf N)$ be a subadditive family with finite non-zero elasticity. Then are equivalent:
\begin{enumerate}[label={\rm (\alph{*})}]
\item\label{it:lem:equivalent_conditions_for_accepted_elasticity(i)} $\mathscr L$ has accepted elasticity.
\item\label{it:lem:equivalent_conditions_for_accepted_elasticity(ii)} There exists $n \in \mathbf N^+$ such that $nk \rho = \rho_{nk}$ for all $k \in \mathbf N^+$.
\item\label{it:lem:equivalent_conditions_for_accepted_elasticity(iii)} $n \rho = \rho_n$ for some $n \in \mathbf N^+$.
\item\label{it:lem:equivalent_conditions_for_accepted_elasticity(iv)} There exists $n \in \mathbf N^+$ such that $n \rho \in \mathbf N^+$ and $nk = \lambda_{nk\rho}$ for all $k \in \mathbf N^+$.
\item\label{it:lem:equivalent_conditions_for_accepted_elasticity(v)} $n \rho \in \mathbf N^+$ and $n = \lambda_{n\rho}$ for some $n \in \mathbf N^+$.
\end{enumerate}
\end{proposition}
\begin{proof}
\ref{it:lem:equivalent_conditions_for_accepted_elasticity(i)} $\Rightarrow$ \ref{it:lem:equivalent_conditions_for_accepted_elasticity(ii)} and \ref{it:lem:equivalent_conditions_for_accepted_elasticity(i)} $\Rightarrow$ \ref{it:lem:equivalent_conditions_for_accepted_elasticity(iv)} follow from Lemma \ref{lem:accepted_elasticity}\ref{it:lem:accepted_elasticity(ii)} (using that $\LLc$ has accepted non-zero elasticity, pick $L \in \mathscr L$ with $\rho(L) = \rho$, notice that $\emptyset \ne L \subseteq \mathbf N^+$ and $\sup L < \infty$, and set $n := \inf L$), while \ref{it:lem:equivalent_conditions_for_accepted_elasticity(ii)} $\Rightarrow$ \ref{it:lem:equivalent_conditions_for_accepted_elasticity(iii)} and \ref{it:lem:equivalent_conditions_for_accepted_elasticity(iv)} $\Rightarrow$ \ref{it:lem:equivalent_conditions_for_accepted_elasticity(v)} are obvious. So, it remains to show that \ref{it:lem:equivalent_conditions_for_accepted_elasticity(iii)} $\Rightarrow$ \ref{it:lem:equivalent_conditions_for_accepted_elasticity(i)} and \ref{it:lem:equivalent_conditions_for_accepted_elasticity(v)} $\Rightarrow$ \ref{it:lem:equivalent_conditions_for_accepted_elasticity(i)}.

\ref{it:lem:equivalent_conditions_for_accepted_elasticity(iii)} $\Rightarrow$ \ref{it:lem:equivalent_conditions_for_accepted_elasticity(i)}: Let $n \in \mathbf N^+$ such that $n \rho = \rho_n$. Since $\rho$ is finite, $\rho_n < \infty$ and $\{n, \rho_n\} \subseteq L$ for some $L \in \mathscr L$. It follows $n \rho = \rho_n \le \sup L$ and $\inf L \le n$, which, by Lemma \ref{lem:inequalities_with_elasticities}\ref{it:lem:inequalities_with_elasticities(iv)}, implies $\rho = \rho(L)$.

\ref{it:lem:equivalent_conditions_for_accepted_elasticity(v)} $\Rightarrow$ \ref{it:lem:equivalent_conditions_for_accepted_elasticity(i)}: Let $n \in \mathbf N^+$ such that $n \rho \in \mathbf N^+$ and $n = \lambda_{n \rho}$. Then $\lambda_{n\rho} < \infty$ and, similarly to the previous analysis, there exists $L \in \mathscr L$ with  $\{\lambda_{n\rho}, n \rho\} \subseteq L$. So $n \rho \le \sup L$ and $\inf L \le \lambda_{n\rho} = n$, which, again by Lemma \ref{lem:inequalities_with_elasticities}\ref{it:lem:inequalities_with_elasticities(iv)}, yields $\rho = \rho(L)$.
\end{proof}
The next two propositions are the key (technical) results of this paper: In particular, the first of them is a substantial im\-prove\-ment of \cite[Lemma 3.4]{Fr-Ge08} (see also Claim 3 in the proof of \cite[Theorem 2.2(2)]{FGKT}).
\begin{proposition}
\label{th:key-theorem_1}
Let $\LLc \subseteq \mathcal P(\mathbf N)$ be a subadditive, primitive family with accepted non-zero elasticity. Then there exists $m \in \mathbf N^+$ such that the following hold:
\begin{enumerate}[label={\rm (\roman{*})}]
\item\label{it:th:key-theorem_1(i)} $\rho_m = m \rho$ and $\lambda_m = m \lambda$.
\item\label{it:th:key-theorem_1(iii)} $\rho_{k+m} = \rho_{k} + m\rho$ and $\lambda_{k+m} = \lambda_{k} + m \lambda$ for all large $k \in \mathbf N$.
\end{enumerate}
\end{proposition}
\begin{proof}
Since $\LLc$ is a primitive family, we get from Lemma \ref{lem:non-emptyness-of-multiples}\ref{it:lem:non-emptyness-of-multiples(ii)} that there is $k_0 \in \mathbf N^+$ for which
\begin{equation}
\label{equ:non-empty-unions}
\UUc_k \ne \emptyset,\ \text{for }k \ge k_0.
\end{equation}
In addition, we infer from Lemmas \ref{lem:inequalities_with_elasticities}\ref{it:lem:inequalities_with_elasticities(ii)} and \ref{lem:accepted_elasticity}\ref{it:lem:accepted_elasticity(ii)}, in view of the fact that $\LLc$ has accepted elasticity, that there exists $n \in \mathbf N^+$ such that $n \rho \in \mathbf N^+$ and
\begin{equation}
\label{equ:formulas-implied-by-accepted-elasticity}
\rho_{nk} = nk \rho
\quad\text{and}\quad
\lambda_{nk\rho} = nk, \text{ for all }k \in \mathbf N^+.
\end{equation}
On the other hand,
Lemma \ref{lem:inequalities_with_elasticities}\ref{it:lem:inequalities_with_elasticities(iii)} gives
\begin{equation}
\label{equ:basic-inequality-on-local-elas}
\rho_{k} \le k\rho < \infty
\quad\text{and}\quad
k\lambda \le \lambda_{k}, \ \text{for all }k \in \mathbf N^+.
\end{equation}
Set $m := k_0 \lcm(n, n\rho)$
and pick $r \in \llb 0, m-1 \rrb$.
Since $\lambda\rho = 1$, we obtain from \eqref{equ:basic-inequality-on-local-elas} that
$$
\rho_{mk + r} -  mk \rho \le r \rho \le (m-1)\rho < \infty
\quad\text{and}\quad
\lambda_{mk+r} - mk\lambda \ge r\lambda \ge 0, \ \text{for all }k \in \mathbf N^+.
$$
This shows that the sets
$\mathcal U_r := \{\rho_{mk+r} -  mk\rho: k \in \mathbf N^+\} \subseteq \mathbf Z
$ and $\mathcal L_r := \{\lambda_{mk+r} - mk\lambda: k \in \mathbf N^+\} \subseteq \mathbf N$ have, respectively, a maximum and a minimum: Let $h_r, \ell_r \in \mathbf N^+$ such that
\begin{equation}
\label{equ:max&min}
{\rho_{mh_r+r} -  mh_r\rho = \sup\fixed[0.3]{\text{ }}\mathcal{U}_r} \in \mathbf Z
\quad\text{and}\quad
\lambda_{m\ell_r+r} -  m\ell_r\lambda = \inf \mathcal{L}_r \in \mathbf N.
\end{equation}
Then, considering that $m \ge k_0$, we derive from \eqref{equ:non-empty-unions} and Lemma \ref{prop:basic(1)}\ref{it:prop:basic(1)(v)} that, for every $k \in \mathbf N^+$,
\begin{equation*}
\label{equ:last_equ}
\begin{split}
\rho_{m(k+h_r)+r} - m(k+h_r)\rho
    & \fixed[-0.25]{\text{ }}\stackrel{\eqref{equ:max&min}}{\le}\fixed[-0.25]{\text{ }}
    \rho_{mh_r+r} - mh_r \rho
    \fixed[-0.25]{\text{ }}\stackrel{\eqref{equ:formulas-implied-by-accepted-elasticity}}{=}\fixed[-0.25]{\text{ }}
    \rho_{mh_r+r} + \rho_{mk} - m(k+h_r) \rho \\
    & \le \rho_{m(k+h_r)+r} - m(k+h_r) \rho,
\end{split}
\end{equation*}
and, in a similar way (note that $m\lambda$ is a positive integer and $mk = mk \lambda \rho$),
\begin{equation*}
\begin{split}
\lambda_{m(k+\ell_r)+r} - m(k+\ell_r) \lambda
    & \fixed[-0.25]{\text{ }}\stackrel{\eqref{equ:max&min}}{\ge}\fixed[-0.25]{\text{ }}
    \lambda_{m\ell_r+r} - m\ell_r \lambda
    \fixed[-0.25]{\text{ }}\stackrel{\eqref{equ:formulas-implied-by-accepted-elasticity}}{=}\fixed[-0.25]{\text{ }}
    \lambda_{m\ell_r+r} + \lambda_{mk} - m(k+\ell_r) \lambda \\
    & \ge \lambda_{m(k+\ell_r)+r} - m(k+\ell_r) \lambda.
\end{split}
\end{equation*}
To wit, we have established that
$$
\rho_{m(k+h_r)+r} = mk\rho + \rho_{mh_r+r}
\quad\text{and}\quad
\lambda_{m(k+\ell_r)+r} = mk\lambda + \lambda_{m\ell_r+r}, \text{ for all }k \in \mathbf N.
$$
It follows that, for every $k \in \mathbf N$ and $\eta \in \mathbf N^+$,
\begin{equation}
\label{equ:stroner-periodicity-on-the-upper-elasticities}
\begin{split}
\rho_{m(k+\eta h_r)+r}
   = mk \rho + m(\eta-1)h_r\rho + \rho_{mh_r+r} = mk \rho + \rho_{m\eta h_r+r}
\end{split}
\end{equation}
and
\begin{equation}
\label{equ:stroner-periodicity-on-the-lower-elasticities}
\begin{split}
\lambda_{m(k+\eta\ell_r)+r}
   = mk \lambda + m(\eta-1)\ell_r\lambda + \lambda_{m\ell_r+r} = mk\lambda + \lambda_{m\eta\ell_r+r}.
\end{split}
\end{equation}
Take
$
s := \lcm(h_0, \ell_0, \ldots, h_{m-1}, \ell_{m-1}) \in \mathbf N^+$. Then,
for each $r \in \llb 0, m-1 \rrb$, there exist $u_r, v_r \in \mathbf N^+$ with $s = h_r u_r = \ell_r v_r$, and we conclude from \eqref{equ:stroner-periodicity-on-the-upper-elasticities} and \eqref{equ:stroner-periodicity-on-the-lower-elasticities} that
\begin{equation}
\label{equ:lifting-to-lcm}
\rho_{m(k+s) + r} = mk\rho + \rho_{ms+r}
\quad\text{and}\quad
\lambda_{m(k+s)+r} = mk\lambda + \lambda_{ms+r},\ \text{for all }k \in \mathbf N.
\end{equation}
With all the above in place, it is now clear from \eqref{equ:formulas-implied-by-accepted-elasticity}, since $m = k_0 \lcm(n, n \rho)$, that $\rho_m = m\rho$ and $\lambda_m = m\lambda$ (recall that $\lambda \rho = 1$).
So, we are only left to prove \ref{it:th:key-theorem_1(iii)}. To this end, let $\kappa$ be an integer $\ge ms$. Then, we can write $\kappa = mk + r$ for some $k \ge s$ and $r \in \llb 0, m-1 \rrb$, and we get from \eqref{equ:lifting-to-lcm} that
$$
\rho_{\kappa + m} = \rho_{m(k+1)+r} = m(k+1-s)\rho + \rho_{ms+r} = m\rho + \rho_{mk+r} = m\rho + \rho_{\kappa}.
$$
Likewise (we omit details), we have $\lambda_{\kappa + m} = m\lambda + \lambda_{\kappa}$, and we are done.
\end{proof}
\begin{proposition}
\label{prop:periodicity(1)}
Assume $\LLc \subseteq \mathcal P(\mathbf N)$ is a subadditive, primitive family with $\Delta(\LLc) \ne \emptyset$ and ac\-cept\-ed elasticity.
Then there exists $m \in \mathbf N^+$ such that, for each $i \in \mathbf N^+$, the following hold for all large $k \in \mathbf N$:
\begin{enumerate}[label={\rm (\roman{*})}]
\item\label{it:cor:periodicity(1)(i)} $\rho_{k+m} - \rho_{k+m,i} = \rho_{k} - \rho_{k,i}$ and $\lambda_{k+m} - \lambda_{k+m,i} = \lambda_{k} - \lambda_{k,i}$.
\item\label{it:cor:periodicity(1)(ii)} $\rho_{k+m,i} - \rho_{k,i} = m \rho$ and $\lambda_{k+m,i} - \lambda_{k,i} = m\lambda$.
\end{enumerate}
\end{proposition}
\begin{proof}
Since $\DeltaF{\LLc}$ is non-empty, $\rho$ is non-zero. So, taking into account that $\LLc$ has accepted elasticity, we get from Proposition \ref{th:key-theorem_1} that there exists $m \in \mathbf N^+$ such that
\begin{equation}
\label{equ:linking-relations}
\rho_{k+m} = \rho_{k} + \rho_m = \rho_{k} + m\rho
\quad\text{and}\quad
\lambda_{k+m} = \lambda_{k} + \lambda_m = \lambda_{k} + m\lambda, \text{ for every large }k.
\end{equation}
Accordingly, fix $i \in \mathbf N^+$. By Lemma \ref{lem:non-emptyness-of-multiples}\ref{it:lem:non-emptyness-of-multiples(iv)}, we have that
\begin{equation}
\label{equ:inequality-with-non-empty-delta-set}
\UUc_k \ne \emptyset
\quad\text{and}\quad
\lambda_{k,1} \le \cdots \le \lambda_{k,i} \le k \le \rho_{k,i} \le \cdots \le \rho_{k,1} < \infty, \text{ for all but finitely many }k.
\end{equation}
It follows by Lemma \ref{prop:basic(1)}\ref{it:prop:basic(1)(v)} and \eqref{equ:linking-relations} that, from some $k$ on,
$$
\rho_{k+m,i} \ge \rho_{k,i} + \rho_m = \rho_{k,i} + \rho_{k+m} - \rho_{k}
\quad\text{and}\quad
\lambda_{k+m,i} \le \lambda_{k,i} + \lambda_m = \lambda_{k,i} + \lambda_{k+m} - \lambda_{k},
$$
which, after rearrangement, leads to
\begin{equation}
\label{equ:non-increasing}
0 \le \rho_{k+m} - \rho_{k+m,i} \le \rho_{k} - \rho_{k,i}
\quad\text{and}\quad
0 \le \lambda_{k+m,i} - \lambda_{k+m} \le \lambda_{k,i} - \lambda_{k}.
\end{equation}
With this in hand, we proceed to prove points \ref{it:cor:periodicity(1)(i)} and \ref{it:cor:periodicity(1)(ii)}.

\ref{it:cor:periodicity(1)(i)} We obtain from \eqref{equ:non-increasing} that there exists $k_i \in \mathbf N$ such that, for every $k \ge k_i$, the $\mathbf N$-valued sequences $(\rho_{k+mh} - \rho_{k+mh,i})_{h \ge 0}$ and $(\lambda_{k+mh} - \lambda_{k+mh,i})_{h \ge 0}$ are both eventually non-increasing, hence eventually constant. In particular, for each $r \in \llb 0, m - 1 \rrb$ there is $h_r \in \mathbf N$ such that, for $h \ge h_r$,
\begin{equation}
\label{equ:quasi-periodicity-of-upper-elasticities}
\rho_{k_i+r+mh} - \rho_{k_i+r+mh,i} = \rho_{k_i+r+mh_r} - \rho_{k_i+r+mh_r,i}
\end{equation}
and
\begin{equation}
\label{equ:quasi-periodicity-of-lower-elasticities}
\lambda_{k_i+r+mh} - \lambda_{k_i+r+mh,i} = \lambda_{k_i+r+mh_r} - \lambda_{k_i+r+mh_r,i}\fixed[0.3]{\text{ }}.
\end{equation}
Now, let $k \ge k_i + m \max(h_0, \ldots, h_{m-1})$. Then, there are uniquely de\-ter\-mined $\kappa \in \mathbf N$ and $r \in \llb 0, m - 1\rrb$ such that $k-k_i = m\kappa + r$, and it is easily seen that $\kappa \ge h_r$. So, we derive from \eqref{equ:quasi-periodicity-of-upper-elasticities} that
\begin{equation*}
\begin{split}
\rho_{k+m} - \rho_{k+m,i}
    & = \rho_{k_i+r+m(\kappa+1)} - \rho_{k_i+r+m(\kappa+1),i} = \rho_{k_i+r+mh_r} - \rho_{k_i+r+mh_r,i} \\
    & = \rho_{k_i+r+m\kappa} - \rho_{k_i+r+m\kappa,i} = \rho_k - \rho_{k,i}\fixed[0.3]{\text{ }},
\end{split}
\end{equation*}
and in a similar way (we omit details) we derive from \eqref{equ:quasi-periodicity-of-lower-elasticities} that
$\lambda_{k+m} - \lambda_{k+m,i} = \lambda_{k} - \lambda_{k,i}\fixed[0.3]{\text{ }}$.

\ref{it:cor:periodicity(1)(ii)} We infer from  \eqref{equ:inequality-with-non-empty-delta-set} and point \ref{it:cor:periodicity(1)(i)} that $\rho_{k+m,i} - \rho_{k,i} = \rho_{k+m} - \rho_{k}$ and $\lambda_{k+m,i} - \lambda_{k,i} = \lambda_{k+m} - \lambda_{k}$ for all large $k$, which, by \eqref{equ:linking-relations}, is enough to conclude.
\end{proof}
\begin{theorem}
\label{cor:main-corollary}
Let $\LLc \subseteq \mathcal P(\mathbf N)$ be a subadditive, primitive family with accepted elasticity. Then there exists $\mu \in \mathbf N^+$ such that, for every $M \in \mathbf N$, the following hold for all but finitely many $k$:
\begin{enumerate}[label={\rm (\roman{*})}]
\item\label{it:cor:main-corollary(i)} $(\rho_{k+\mu} - \mathscr{U}_{k+\mu}) \cap \llb 0, M \rrb = (\rho_k - \mathscr{U}_k) \cap \llb  0,M\rrb$.
\item\label{it:cor:main-corollary(ii)} $(\mathscr{U}_{k+\mu} - \lambda_{k+\mu}) \cap \llb 0, M \rrb = (\mathscr{U}_k - \lambda_k) \cap \llb 0, M \rrb$.
\end{enumerate}
\end{theorem}
\begin{proof}
We distinguish two cases, depending on whether the set of distances of $\LLc$ is empty.
\vskip 0.1cm
\textsc{Case 1:} $\DeltaF{\LLc} = \emptyset$. We infer from Lemma \ref{lem:non-emptyness-of-multiples}\ref{it:lem:non-emptyness-of-multiples(ii)} and our assumptions that $\UUc_{k} - \rho_{k} = \UUc_{k} - \lambda_{k} = \{0\}$ for all large $k$. Whence the conclusion is trivial (with $\mu := 1$).
\vskip 0.1cm
\textsc{Case 2:} $\DeltaF{\LLc} \ne \emptyset$.
By Proposition \ref{prop:periodicity(1)}, we can find an integer $m \ge 1$ with the property that, for every $i \in \mathbf N^+$, there is $\kappa_i \in \mathbf N$ such that, for all $k \ge \kappa_i$ and each $j \in \llb 1, i \rrb$,
\begin{equation}
\label{equ:another-brick-in-the-wall}
\rho_{k+m} - \rho_{k+m,j} = \rho_{k} - \rho_{k,j}
\quad\text{and}\quad
\lambda_{k+m} - \lambda_{k+m,j} = \lambda_{k} - \lambda_{k,j}\fixed[0.2]{\text{ }}.
\end{equation}
Now, fix $M \in \mathbf N$. By Lemma \ref{lem:non-emptyness-of-multiples}\ref{it:lem:non-emptyness-of-multiples(iv)}, there exists $k_M \ge \kappa_{M+1}$ such that $\UUc_{k,M+1} \ne \emptyset$ for $k \ge k_M$, which, together with \eqref{equ:another-brick-in-the-wall}, shows that, for all large $k$,
$
(\rho_{k} - \UUc_{k}) \cap \llb 0, M \rrb = (\rho_{k+m} - \UUc_{k+m}) \cap \llb 0, M \rrb$
and $
(\UUc_{k} - \lambda_{k}) \cap \llb 0, M \rrb = (\UUc_{k+m} - \lambda_{k+m}) \cap \llb 0, M \rrb$. This finishes the proof (with $\mu := m$).
\end{proof}
As was already mentioned, our main goal in the present work is to understand the structure of the unions $\UUc_k(\LLc)$ when $\LLc$ is a suitable collection of subsets of $\mathbf N$. To this end, we make the following:
\begin{definition}
\label{def:structure-theorem}
A family $\mathscr L \subseteq \mathcal P(\mathbf N)$ satisfies the \textit{Structure Theorem for Unions} if there are $d \in \mathbf N^+$ and $M \in \mathbf N$ such that
$(k + d \cdot \mathbf Z) \cap \llb \lambda_k + M, \rho_k - M \rrb \subseteq \mathscr U_k \subseteq k + d \cdot \mathbf Z
$ for all large $k \in \mathbf N$.
\end{definition}
Concretely, we will prove a characterization of when the Structure Theorem for Unions holds in the case $\mathscr L \subseteq \mathcal P(\mathbf N)$ is a subadditive family (Theorem \ref{th:characterization-of-structure-theorem}). But first, we need some preliminaries.
\begin{lemma}
\label{lem:delta-set-of-unions}
Let $\LLc \subseteq \mathcal P(\mathbf N)$ be a subadditive family. The following hold:
\begin{enumerate}[label={\rm (\roman{*})}]
\item\label{it:lem:delta-set-of-unions(i)} $\sup \Delta_{\cup}(\LLc) \le \sup \DeltaF{\LLc}$.
\item\label{it:lem:delta-set-of-unions(ii)} $\delta = \inf \fixed[-0.4]{\text{ }}\{\inf \Delta(\UUc_k): k \ge k_0\}$ for every $k_0 \in \mathbf N$. In particular, $\delta = \inf \Delta_{\cup}(\LLc)$.
\item\label{it:lem:delta-set-of-unions(iii)} If $\Delta(\LLc) \ne \emptyset$, then $\Delta_\cup(\LLc) \ne \emptyset$ and $\inf \Delta_\cup(\LLc) = \gcd \Delta_\cup(\LLc) = \gcd \Delta(\LLc) = \delta$.
\end{enumerate}
\end{lemma}
\begin{proof}
\ref{it:lem:delta-set-of-unions(i)} Pick $k \in \mathbf N$. It suffices to show that $\sup \Delta(\UUc_k) \le \sup \DeltaF{\LLc}$. If $\Delta(\UUc_k)$ is empty, this is obvious. Otherwise, let $d \in \Delta(\UUc_k)$. Then, there exists $x \in \mathbf N$ such that $\UUc_k \cap \llb x, x + d \fixed[0.2]{\text{ }}\rrb = \{x, x + d\}$, whence it is clear that $k \le x$ or $x+d \le k$. Accordingly, we can find $L \in \mathscr L$ such that either  $\{k, x + d\} \subseteq L$ (if $k \le x$) or $\{x, k\} \subseteq L$ (if $x+d \le k$). It follows
$$
L \cap \llb x, x+d \fixed[0.2]{\text{ }} \rrb \subseteq \UUc_k \cap \llb x, x+d \fixed[0.2]{\text{ }}\rrb = \{x, x+d\},
$$
which gives $d \le \sup \Delta(L) \leq \sup \Delta(\LLc)$ and leads to the desired inequality.

\ref{it:lem:delta-set-of-unions(ii)} Fix $k_0 \in \mathbf N$, and set $\delta_{k_0} := \inf \fixed[-0.4]{\text{ }}\{\inf \Delta(\UUc_k): k \ge k_0\}$. By Lemma \ref{prop:basic(1)}\ref{it:prop:basic(1)(ii)}, $\DeltaF{\LLc} = \emptyset$ if and only if $\Delta(\UUc_k) = \emptyset$ for all $k$. So, if $\Delta(\LLc)$ is empty, the conclusion is trivial, because $\delta = \delta_{k_0} = \infty$. Consequently, we assume from now on that $\Delta(\LLc) \ne \emptyset$.

Then $\delta \in \mathbf N^+$ and $\delta_{k_0} = \inf \Delta(\UUc_{\kappa_0}) < \infty$ for some $\kappa_0 \ge k_0$, which, in turn, implies that there is $x \in \mathbf N$ such that $\UUc_{\kappa_0} \cap \llb x, x+\delta_{k_0} \rrb = \{x, x+\delta_{k_0}\}$. By Corollary \ref{cor:delta_sets}\ref{it:cor:delta_sets(ii)}, this yields $\delta \mid \delta_{k_0}$, and hence $\delta \le \delta_{k_0}$.

On the other hand, we get from Corollary \ref{cor:delta_sets}\ref{it:cor:delta_sets(iii)} and Lemma \ref{prop:basic(1)}\ref{it:prop:basic(1)(iv)} that $\ell + \delta \cdot \llb 0, k_0+1 \rrb \subseteq \UUc_{\ell}$ for some integer $\ell \ge k_0$. Thus we obtain $\delta_{k_0} \le \inf \Delta\fixed[-0.1]{\text{ }}\bigl(\UUc_{\ell}\bigr) \le \delta \le \delta_{k_0}$, which completes the proof, insofar as it is straightforward that $\inf \Delta_{\cup}(\LLc) = \inf \fixed[-0.4]{\text{ }}\{\inf \Delta(\UUc_k): k \in \mathbf N\}$.

\ref{it:lem:delta-set-of-unions(iii)} It is enough to prove that $\gcd \Delta_\cup(\LLc) = \inf \Delta_\cup(\LLc)$: The rest will follow from \ref{it:lem:delta-set-of-unions(ii)} and Prop\-o\-si\-tion \ref{prop:generalized-deltas}. For this, assume $\Delta(\LLc) \ne \emptyset$ and set $\mathscr L_{\cup} := \{\UUc_k: k \in \mathbf N\}$. Clearly, $\mathscr L_\cup$ is a subfamily of $\mathcal P(\mathbf N)$ with non-empty set of distances, and we infer from Lemma \ref{prop:basic(1)}\ref{it:prop:basic(1)(iv)} that $\mathscr L_\cup$ is, in fact, subadditive. So, again by Prop\-o\-si\-tion \ref{prop:generalized-deltas}, we have $\gcd \Delta_\cup(\LLc) = \inf \Delta_\cup(\LLc)$.
\end{proof}
\begin{proposition}
\label{prop:the-difference-is-necessarily-delta}
Let $\mathscr L \subseteq \mathcal P(\mathbf N)$ be a subadditive, primitive family with $\Delta(\LLc) \ne \emptyset$, and suppose there exist $M \in \mathbf N$, $d \in \mathbf N^+$, and infinitely many $k$ for which $(k + d \cdot \mathbf Z) \cap \llb \lambda_{k} + M, \rho_{k} - M \rrb \subseteq \UUc_{k} \subseteq k + d \cdot \mathbf Z$. Then $d = \delta$.
\end{proposition}
\begin{proof}
Since $\Delta(\LLc)$ is non-empty, $\delta$ is a positive integer. Moreover, $\LLc$ being a subadditive family implies by Corollary \ref{cor:delta_sets}\ref{it:cor:delta_sets(iii)} that there are $\ell \in \mathbf N^+$ and $L \in \mathscr L$ for which
\begin{equation}
\label{equ:set-with-long-AP-with-difference-delta}
\ell + \delta \cdot \llb 0, M+1 \rrb \subseteq L \subseteq \UUc_\ell.
\end{equation}
Similarly, we obtain from Lemma \ref{lem:non-emptyness-of-multiples}\ref{it:lem:non-emptyness-of-multiples(iv)} that there exists $\kappa_0 \in \mathbf N$ such that
\begin{equation}
\label{equ:various-conditions}
\UUc_{k} \ne \emptyset
\quad\text{and}\quad
\lambda_{k} + (d+2M)\delta \le k \le \rho_{k} - (d+2M)\delta, \text{ for }k \ge \kappa_0.
\end{equation}
So, considering that, by hypothesis, $(k + d \cdot \mathbf Z) \cap \llb \lambda_{k} + M, \rho_{k} - M \rrb \subseteq \UUc_{k} \subseteq k + d \cdot \bf Z$ for infinitely many $k$, we infer from \eqref{equ:various-conditions} that
\begin{equation}
\label{equ:convenient-containment}
\UUc_{k_0} \cap \llb \lambda_{k_0} + M, \rho_{k_0} - M \rrb = k_0 + d \cdot \llb -x_{k_0}, y_{k_0} \rrb,
\end{equation}
for some $k_0 \ge \kappa_0+\ell$ and $x_{k_0}, y_{k_0} \in \mathbf N^+$. It follows $d \in \Delta(\UUc_{k_0}) \subseteq \Delta_\cup(\LLc)$, which, combined with Lemma \ref{lem:delta-set-of-unions}\ref{it:lem:delta-set-of-unions(iii)}, proves $\delta \le d$. Consequently, we are left to show $d \le \delta$.

To this end, note that $k_0 - \ell \in \UUc_{k_0 - \ell}$ (because $k_0 - \ell \ge \kappa_0$, and by construction $\UUc_{k} \ne \emptyset$ for $k \ge \kappa_0$). Therefore, we get from \eqref{equ:set-with-long-AP-with-difference-delta} and Lemma \ref{prop:basic(1)}\ref{it:prop:basic(1)(iv)} that
$$
k_0 + \delta \cdot \llb 0, M+1 \rrb \subseteq k_0 - \ell + \UUc_\ell \subseteq \UUc_{k_0 - \ell} + \UUc_{\ell} \subseteq \UUc_{k_0},
$$
which, together with \eqref{equ:various-conditions} and \eqref{equ:convenient-containment}, yields
$$
k_0 + \delta \cdot \llb 0, M+1 \rrb \subseteq \UUc_{k_0} \cap \llb k_0, \rho_{k_0} - M \rrb \subseteq \UUc_{k_0} \cap \llb \lambda_{k_0} + M, \rho_{k_0} - M \rrb = k_0 + d \cdot \llb x_{k_0}, y_{k_0} \rrb.
$$
In particular, we see from here that $\delta \cdot \llb 0, M+1 \rrb \subseteq d \cdot \mathbf Z$, which is possible only if $d \le \delta$.
\end{proof}
The next result is essentially a revision of \cite[Lemma 2.12]{FGKT}.
\begin{lemma}
\label{lem:reduction_to_the_upper_part}
Let $\LLc \subseteq \mathcal P(\mathbf N)$ be a subadditive, primitive family, and let $d \in \mathbf N^+$.
Then are equivalent:
\begin{enumerate}[label={\rm (\alph{*})}]
\item\label{it:lem:reduction_to_the_upper_part(i)} There is $M \in \mathbf N$ such that $(k + d \cdot \mathbf Z) \cap \llb \lambda_{k} + M, \rho_{k} - M \rrb \subseteq \UUc_{k}$ for all large $k$.
\item\label{it:lem:reduction_to_the_upper_part(iii)} There is $M^{\prime} \in \mathbf N$ such that $(k + d \cdot \mathbf Z) \cap \llb k, \rho_{k} - M^{\prime} \rrb \subseteq \UUc_{k}$ for all large $k$.
\end{enumerate}
\end{lemma}
\begin{proof}
If $\Delta(\LLc)$ is empty, the equivalence of conditions \ref{it:lem:reduction_to_the_upper_part(i)} and \ref{it:lem:reduction_to_the_upper_part(iii)} is trivial, since $\UUc_k \subseteq \{k\}$ for all $k$. So, assume from now on that $\Delta(\LLc)$ is non-empty. Then we get from Lemma \ref{lem:non-emptyness-of-multiples}\ref{it:lem:non-emptyness-of-multiples(iv)} that, for every $i \in \mathbf N^+$, there is $k_i \in \mathbf N$ such that
\begin{equation}
\label{equ:some_conditions}
\UUc_{k, i} \ne \emptyset
\quad\text{and}\quad
\lambda_{k} + id \le k \le \rho_{k} - id, \text{ for }k \ge k_i.
\end{equation}
Based on these premises, we proceed to show that \ref{it:lem:reduction_to_the_upper_part(i)} $\Rightarrow$ \ref{it:lem:reduction_to_the_upper_part(iii)} $\Rightarrow$ \ref{it:lem:reduction_to_the_upper_part(i)}.

\ref{it:lem:reduction_to_the_upper_part(i)} $\Rightarrow$ \ref{it:lem:reduction_to_the_upper_part(iii)}: By hypothesis, $(k + d \cdot \mathbf Z) \cap \llb \lambda_k + M, \rho_k - M \rrb \subseteq \UUc_{k}$ for all large $k$. Also, we have by \eqref{equ:some_conditions} that $\lambda_{k} + M \le k \le \rho_{k} - M$ for $k \ge k_M$. Therefore, it is obvious that $(k + d \cdot \mathbf Z)  \cap \llb k, \rho_{k} - M \rrb \subseteq \UUc_{k}$ for all but finitely many $k$, which is enough to conclude (with $M^\prime := M$).

\ref{it:lem:reduction_to_the_upper_part(iii)} $\Rightarrow$ \ref{it:lem:reduction_to_the_upper_part(i)}:
By assumption, there exists $k_0 \in \mathbf N^+$ such that $(k + d \cdot \mathbf Z) \cap \llb k, \rho_{k} - M^\prime \rrb \subseteq \UUc_{k}$ for $k \ge k_0$. Accordingly, fix $k \ge \max(k_0, k_1)$, and let
$M := \max (k_0, k_1, M^\prime) \in \mathbf N^+$.
It suffices to prove that
$$
\mathscr V_k := (k + d \cdot \mathbf Z) \cap \llb \lambda_{k} + M, k \rrb \subseteq \UUc_{k}.
$$
To this end, notice that, by \eqref{equ:some_conditions},  $\UUc_k \ne \emptyset$ (because $k \ge k_1$) and $\lambda_{k} \in \mathbf N$, and let $q \in \mathscr V_k$. Then $\UUc_{q - \lambda_{k}}$ is non-empty, since $q - \lambda_k \ge M \ge k_1$. In addition, we obtain from Lemma \ref{prop:basic(1)}\ref{it:prop:basic(1)(i)} that $
k \in \mathscr U_{\lambda_{k}} \ne \emptyset$ and $
q \le k \le \rho_{\lambda_{k}}$. Consequently, we infer from Lemma \ref{prop:basic(1)}\ref{it:prop:basic(1)(v)} that
\begin{equation}
\label{equ:yet-another-very-little-effort}
q \le k \le k + M \le k + (q - \lambda_{k}) \le \rho_{\lambda_{k}} + \rho_{q - \lambda_{k}}
\le \rho_q.
\end{equation}
On the other hand, it is clear from the above that $q \ge M \ge k_0$. It follows
$$
(q + d \cdot \mathbf Z) \cap \llb q, \rho_q - M^\prime \rrb \subseteq \UUc_q,
$$
and hence
$k \in \UUc_q$, because $d \mid q - k$ and we have by \eqref{equ:yet-another-very-little-effort} that $q \le k \le \rho_q - M \le \rho_q - M^\prime$. By Lemma \ref{prop:basic(1)}\ref{it:prop:basic(1)(i)}, this implies $q \in \UUc_{k}$. So we are done, since $q \in \mathscr V_k$ was arbitrary.
\end{proof}
\begin{theorem}
\label{th:characterization-of-structure-theorem}
Let $\LLc \subseteq \mathcal P(\mathbf N)$ be a subadditive, primitive family with non-empty set of distances, and let $D := \limsup_k \sup \Delta(\UUc_k)$.
Then the following are equivalent:
\begin{enumerate}[label={\rm (\alph{*})}]
\item\label{it:th:structure_theorem(i)} $\LLc$ satisfies the Structure Theorem for Unions.
\item\label{it:th:structure_theorem(ii)} $D \in \mathbf N^+$ and there exist $\ell \in \mathbf N^+$ and $N \in \mathbf N$ such that $\ell + \delta \cdot \llb 0, D \rrb \subseteq \UUc_\ell$ and $(k + \delta \cdot \mathbf Z) \cap \llb \rho_{k - \ell} + \ell, \rho_k - N \rrb \subseteq \UUc_k$ for all large $k \in \mathbf N$.
\end{enumerate}
In particular, condition \ref{it:th:structure_theorem(ii)} is satisfied if $D \in \mathbf N^+$ and $\rho_k = \infty$ for some $k \in \mathbf N$.
\end{theorem}
\begin{proof}
Since $\Delta(\LLc)$ is non-empty, $\delta$ is a positive integer and, by Proposition \ref{lem:delta-set-of-unions}\ref{it:lem:delta-set-of-unions(iii)}, $\Delta_\cup(\LLc) \subseteq \delta \cdot \mathbf N^+$. Consequently, we see that $\UUc_k \subseteq k + \delta \cdot \mathbf Z$ for all $k \in \mathbf N$. Moreover, $\mathscr L$ is a primitive family, so we obtain from Lemma \ref{lem:non-emptyness-of-multiples}\ref{it:lem:non-emptyness-of-multiples(iv)} that, for each $i \in \mathbf N^+$, there exists $k_i \in \mathbf N$ such that
\begin{equation}
\label{equ:unions-for-large-ks}
\UUc_{k,i} \ne \emptyset\quad\text{and}\quad\lambda_{k,1} \le \cdots \le \lambda_{k,i} \le k \le \rho_{k,i} \le \cdots \le \rho_{k,1}, \text{ for }k \ge k_i.
\end{equation}
With these preliminaries in mind, we proceed to demonstrate that \ref{it:th:structure_theorem(i)} $\Rightarrow$ \ref{it:th:structure_theorem(ii)} $\Rightarrow$ \ref{it:th:structure_theorem(i)}, while noting that the ``In particular'' part of the statement is a trivial consequence of Lemma \ref{lem:non-emptyness-of-multiples}\ref{it:lem:non-emptyness-of-multiples(iiiq)}.

\ref{it:th:structure_theorem(i)} $\Rightarrow$ \ref{it:th:structure_theorem(ii)}: By hypothesis (and Definition \ref{def:structure-theorem}), there are $M \in \mathbf N$ and $d \in \mathbf N^+$ such that
\begin{equation}
\label{equ:containment-condition}
(k + d \cdot \mathbf Z) \cap \llb \lambda_{k} + M, \rho_{k} - M \rrb \subseteq \UUc_{k} \subseteq k + d \cdot \mathbf Z, \text{ for all large }k.
\end{equation}
It follows by Proposition \ref{prop:the-difference-is-necessarily-delta} that $d = \delta$, and hence by Lemma \ref{lem:delta-set-of-unions}\ref{it:lem:delta-set-of-unions(ii)} that
$$
1 \le \delta \le \inf \Delta(\UUc_k) \le \sup \Delta(\UUc_k) \le M + \delta, \text{ for all but finitely many }k.
$$
In particular, this shows that $D \in \mathbf N^+$. Accordingly, let $\ell \in \mathbf N^+$ such that $\ell + \delta \cdot \llb 0, D \rrb \subseteq \UUc_\ell$, and take $\mu := 1 + M + \ell$ (note that the existence of such an $\ell$ is guaranteed by Corollary \textup{\ref{cor:delta_sets}}\ref{it:cor:delta_sets(iii)} and the finiteness of the limit $D$). Then $\mu \in \mathbf N^+$, and we derive from \eqref{equ:unions-for-large-ks} that
$$
\lambda_k + M < \lambda_k + M + \ell
\le
k = (k - \ell) + \ell
\le \rho_{k-\ell}, \text{ for } k \ge k_\mu + \ell.
$$
Therefore, we find that
$$
(k + \delta \cdot \mathbf Z) \cap \llb \rho_{k-\ell} + \ell, \rho_k - M \rrb \subseteq (k + \delta \cdot \mathbf Z) \cap \llb \lambda_k + M, \rho_k - M \rrb \fixed[-0.85]{\text{ }}\stackrel{\eqref{equ:containment-condition}}{\subseteq}\fixed[-0.85]{\text{ }} \mathscr U_k,
$$
which proves the claim with $N := M$.

\ref{it:th:structure_theorem(ii)} $\Rightarrow$ \ref{it:th:structure_theorem(i)}: Let $\ell \in \mathbf N^+$ have the property that
$
\UUc^* := \ell + \delta \cdot \fixed[-0.1]{\text{ }} \llb 0, D \rrb \subseteq \UUc_{\ell}$ (recall Corollary \textup{\ref{cor:delta_sets}}\ref{it:cor:delta_sets(iii)}). Then we get from \eqref{equ:unions-for-large-ks} and Lemma \ref{prop:basic(1)}\ref{it:prop:basic(1)(iv)} that
\begin{equation}
\label{equ:inclusion(1)}
k + \delta \cdot \fixed[-0.1]{\text{ }} \llb 0, D \rrb
\subseteq \UUc^* + \UUc_{k - \ell} \subseteq \UUc_{\ell} + \UUc_{k-\ell} \subseteq \UUc_k, \text{ for }k \ge k_1 + \ell.
\end{equation}
On the other hand, it follows from our assumptions that there exist $k_0 \in \mathbf N^+$ and $N \in \mathbf N$ for which
\begin{equation}
\label{equ:inclusion(2)}
\sup \Delta(\UUc_k) \le D\quad\text{and}\quad
\mathscr P_k := (k + \delta \cdot \mathbf Z) \cap \llb \rho_{k- \ell} + \ell, \rho_k - N \rrb \subseteq \UUc_k, \text{ for }k \ge k_0.
\end{equation}
Fix $k \ge \ell + \max(k_0, k_1, k_{N+1})$ and set $\UUc_k^\ast := \UUc^* + \UUc_{k-\ell}$.
Then $\sup \Delta(\mathscr U_{k-\ell}) \le D$, and because $\mathscr U^\ast$ is an AP with difference $\delta$ and $|\mathscr U^\ast| = D+1$, it is clear that
$\UUc_k^\ast$ is also an AP with difference $\delta$, i.e.,
\begin{equation}
\label{equ:a-longer-AP}
\UUc_k^\ast = \llb \inf \UUc_k^\ast , \sup \UUc_k^\ast  \rrb \cap (\inf \UUc_k^\ast  + \delta \cdot \mathbf Z) \fixed[-0.85]{\text{ }}\stackrel{\eqref{equ:inclusion(1)}}{\subseteq}\fixed[-0.85]{\text{ }} \mathscr U_k.
\end{equation}
Moreover, we have that
\begin{equation}
\label{equ:disuguaglianze-varie}
k \in \UUc_k^\ast,
\quad
\inf \UUc_k^\ast = \ell + \lambda_{k-\ell},
\quad\text{and}\quad
\sup \UUc_k^* =
(\ell + \delta D) + \rho_{k-\ell} \ge \ell + \rho_{k-\ell}.
\end{equation}
But $\ell + \lambda_{k-\ell} \le k = \ell + (k - \ell) \le \ell + \rho_{k - \ell}$, and therefore it is straightforward that
$$
(k + \delta \cdot \mathbf Z) \cap \llb k, \rho_{k} - N \rrb \subseteq \bigl((k + \delta \cdot \mathbf Z) \cap \llb \ell + \lambda_{k-\ell}, \ell + \rho_{k - \ell} \rrb\bigr) \cup \bigl((k + \delta \cdot \mathbf Z) \cap \llb \ell + \rho_{k - \ell}, \rho_{k} - N\rrb\bigr).
$$
So, we infer from  \eqref{equ:inclusion(2)}-\eqref{equ:disuguaglianze-varie} that $(k + \delta \cdot \mathbf Z) \cap \llb k, \rho_{k} - N \rrb
\subseteq \UUc_k^* \cup \mathscr P_k \subseteq \UUc_k$, which implies, by Lemma \ref{lem:reduction_to_the_upper_part}, that $\LLc$ satisfies the Structure Theorem for Unions.
\end{proof}
\begin{remark}
	\label{remark:finiteness-of-delta-set-implies-finiteness-of-Udelta-set}
	Theorem \ref{th:characterization-of-structure-theorem} is a \textit{proper} generalization of \cite[Theorem 2.2(1)]{FGKT}. The latter applies, in fact, to the case when $\LLc$ is a directed subfamily of $\mathcal P(\mathbf N)$ for which $\Delta(\LLc)$ is finite (and non-empty). But we know from Lemma \ref{lem:delta-set-of-unions}\ref{it:lem:delta-set-of-unions(i)} that $\sup \Delta_\cup(\LLc) \le \sup \Delta(\LLc)$, and condition \ref{it:th:structure_theorem(ii)} in Theorem \ref{th:characterization-of-structure-theorem} is definitely weaker than the finiteness of the set of distances: E.g., if $L := \{2^k: k \in \mathbf N\}$, then $\{\mathbf N_{\ge 2}, L\} \subseteq \mathcal P(\mathbf N)$ is a directed family with $\sup \Delta(\UUc_k) = 1$ for $k \ge 2$, but $\sup \Delta(L) = \infty$ (a much more interesting example in the same vein will be discussed at the end of \S{ }\ref{sec:focus-on-sets-of-lengths}).
\end{remark}
Now we look for sufficient conditions under which Theorem \ref{th:characterization-of-structure-theorem} can be used to show that a subadditive subfamily of $\mathcal P(\mathbf N)$ satisfies the Structure Theorem for Unions. We start with a couple of lemmas.
\begin{lemma}
	\label{lem:growth-rate-of-local-elasticities-in-weakly-directed-families}
	Let $\LLc \subseteq \mathcal P(\mathbf N)$ be a subadditive, primitive family. Then are equivalent:
	\begin{enumerate}[label={\rm (\alph{*})}]
		\item\label{it:lem:growth-rate-of-local-elasticities-in-weakly-directed-families(a)} There is $K \in \mathbf N$ such that $\rho_{k+1} \le \rho_{k} + K$ \textup{(}respectively, $\lambda_{k} - K \le \lambda_{k+1}$\textup{)} for all large $k$.
		\item\label{it:lem:growth-rate-of-local-elasticities-in-weakly-directed-families(b)} There are $q, N \in \mathbf N^+$ such that $\rho_{k+q} \le \rho_{k} + N$ \textup{(}respectively, $\lambda_{k} - N \le \lambda_{k+q}$\textup{)} for all large $k$.
	\end{enumerate}
\end{lemma}
\begin{proof}
	 \ref{it:lem:growth-rate-of-local-elasticities-in-weakly-directed-families(a)} $\Rightarrow$ \ref{it:lem:growth-rate-of-local-elasticities-in-weakly-directed-families(b)} is obvious. As for the other direction, assume there exist $k_0 \in \mathbf N$ and $q, N \in \mathbf N^+$ such that
	$
	\rho_{k+q} \le \rho_{k} + N$ (respectively, $
	\lambda_{k} - N \le \lambda_{k+q}$) for $k \ge k_0$.
	Then, it is found (by induction) that
	\begin{equation}
	\label{equ:long-range-inequality}
	\rho_{k + qh} \le \rho_{k} + hN
	\quad
	\text{(respectively, }\lambda_{k} - hN \le \lambda_{k+qh}\text{)},
	\text{ for all }h \in \mathbf N \text{ and } k \ge k_0.
	\end{equation}
	Moreover, we know from Lemma \ref{lem:non-emptyness-of-multiples}\ref{it:lem:non-emptyness-of-multiples(ii)} that there is $k_1 \in \mathbf N^+$ such that $\UUc_{k} \ne \emptyset$ for $k \ge k_1$.
	Accordingly, set
	$K := 2k_1 N + \lambda_{2qk_1 - 1} \in \mathbf N$. Then we get from Lemma \ref{prop:basic(1)}\ref{it:prop:basic(1)(v)} that, for $k \ge \max(k_0, k_1)$,
	\begin{gather*}
	\rho_{k+1} \le \rho_{k+1} + \rho_{2qk_1 - 1} \le \rho_{k + 2qk_1}
	\fixed[-0.85]{\text{ }}\stackrel{\eqref{equ:long-range-inequality}}{\le}\fixed[-0.85]{\text{ }}
	\rho_{k} + 2 k_1 N \le \rho_{k} + K \\
	\text{(respectively, }\lambda_{k} - K = (\lambda_{k} - 2 k_1 N) - \lambda_{2qk_1 - 1}
	\fixed[-0.85]{\text{ }}\stackrel{\eqref{equ:long-range-inequality}}{\le}\fixed[-0.85]{\text{ }}
	\lambda_{k + 2qk_1} - \lambda_{2qk_1 - 1} \le \lambda_{k+1}.\text{)}
	\qedhere
	\end{gather*}
\end{proof}
\begin{lemma}
	\label{lem:basic-properties-of-Delta}
	Let $L, L^\prime \subseteq \mathbf N$. The following hold:
	\begin{enumerate}[label={\rm (\roman{*})}]
		\item\label{it:lem:basic-properties-of-Delta(i)} If $\inf L = \inf L^\prime$, $\sup L = \sup L^\prime$, and $L \subseteq L^\prime$, then $\sup \Delta(L^\prime) \le \sup \Delta(L)$.
		\item\label{it:lem:basic-properties-of-Delta(ii)} $\sup \Delta(L + L^\prime) \le \max(\sup \Delta(L), \sup \Delta(L^\prime))$.
	\end{enumerate}
\end{lemma}
\begin{proof}
	\ref{it:lem:basic-properties-of-Delta(i)} If $\Delta(L^\prime) = \emptyset$, then $\sup \Delta(L^\prime) = 0$ and there is nothing left to prove. Otherwise, let $d \in \Delta(L^\prime)$: It suffices to prove $d \le \sup \Delta(L)$. For, pick $\ell \in \mathbf N$ such that $L^\prime \cap \llb \ell, \ell + d \fixed[0.2]{\text{ }} \rrb = \{\ell, \ell + d\}$. Accordingly, let $x := \sup \bigl(L \cap \llb 0, \ell \rrb\bigr)$ and $y := \inf \bigl( L \cap \llb \ell+1, \infty \rrb\bigr)$. It is clear that $x, y \in L$, since our assumptions imply that $
	\inf L = \inf L^\prime \le \ell < \ell + d \le \sup L^\prime = \sup L$.
	It follows $d \le y - x \in \Delta(L)$, because $L \subseteq L^\prime$ and there exists no element in $L^\prime$ that is strictly in between $\ell$ and $\ell + d$. Thus, we obtain $d \le \sup \Delta(L)$.
	
	\ref{it:lem:basic-properties-of-Delta(ii)} If $\Delta(L+L^\prime)$ is empty, the conclusion is trivial. Otherwise, pick $d \in \Delta(L + L^\prime)$, and let $x, y \in L$ and $x^\prime, y^\prime\in L^\prime$ such that $(L + L^\prime) \cap \llb x+x^\prime,y + y^\prime \rrb = \{x+x^\prime,y + y^\prime\}$ and $d = (y + y^\prime) - (x+x^\prime) \ge 1$.
	
	Now, using that $\Delta(X+k) = \Delta(X)$ for all $X \subseteq \mathbf Z$ and $k \in \mathbf Z$, we can assume without loss of generality that $x = x^\prime = 0$. It follows (up to symmetry) that $y \ge 1$. Accordingly, set $z := \inf L^+$.
	
	We derive from the above that $z \in \Delta(L) \cap L^+ \cap (L + L^\prime)$, and since $(L + L^\prime) \cap \llb 0,y + y^\prime \rrb = \{0,d\}$, we conclude that $d \le z \le \sup \Delta(L)$. This finishes the proof, because $d \in \Delta(L+L^\prime)$ was arbitrary.
\end{proof}
With this in hand, we first prove a generalization (from directed to subadditive families) of a remark made in the comments after the statement of \cite[Theorem 2.2(1)]{FGKT}, and then a result showing how ``natural restrictions'' on the growth rate of the upper and lower local elasticities are enough by themselves to imply the Structure Theorem for Unions.
\begin{corollary}
	\label{th:finite-delta-set-bound-on-growth-rate-of-rho_k}
	Let $\mathscr L \subseteq \mathcal P(\mathbf N)$ be a subadditive, primitive family for which $\Delta(\mathscr L)$ is finite and there is $K \in \mathbf N$ such that $\rho_{k+1} \le \rho_k + K$ for all large $k$. Then $\mathscr L$ satisfies the Structure Theorem for Unions.
\end{corollary}
\begin{proof}
	Let $D := \limsup_k \sup \Delta(\mathscr U_k)$. If $\Delta(\LLc)$ is empty, then $\mathscr U_k \subseteq \{k\}$ for all $k$ and the conclusion is trivial. Therefore, we suppose from here on that $\Delta(\LLc) \ne \emptyset$.
	
	We have from Lemma \ref{lem:non-emptyness-of-multiples}\ref{it:lem:non-emptyness-of-multiples(iv)} that there exists $k_0 \in \mathbf N$ such that $\mathscr U_k \ne \emptyset$ for $k \ge k_0$; and from points \ref{it:lem:delta-set-of-unions(i)} and \ref{it:lem:delta-set-of-unions(ii)} of Lemma \ref{lem:delta-set-of-unions} that
	$1 \le \delta \le D \le \sup \Delta_\cup(\LLc) \le \sup \Delta(\LLc) < \infty$. So, $D$ is a positive integer, and we get from Corollary \textup{\ref{cor:delta_sets}}\ref{it:cor:delta_sets(iii)} that $\ell + \delta \cdot \llb 0, D \rrb \subseteq \mathscr U_\ell$ for some $\ell \in \mathbf N^+$.
	
	By Theorem \ref{th:characterization-of-structure-theorem}, it is hence enough to show that there exists $N \in \mathbf N$ such that the interval $\llb \rho_{k-\ell} + \ell, \rho_k - N \rrb$ is empty for all but finitely many $k$. But this is now straightforward: If $\rho_k < \infty$ for all $k \ge k_0$, we take $N := \ell K$ and note that, by the hypothesis and the above,
	$$
	\rho_k - \rho_{k-\ell} = \sum_{i=k-\ell}^{k-1} (\rho_{i+1} - \rho_i) \le \ell K, \ \text{for }k \ge k_0 + \ell;
	$$
	otherwise, it follows by Lemma \ref{lem:non-emptyness-of-multiples}\ref{it:lem:non-emptyness-of-multiples(iiiq)} that $\rho_k = \infty$ for all large $k$, and hence we can take $N := 0$.
\end{proof}
\begin{theorem}
\label{th:structure-theorem-for-unions}
Let $\mathscr L \subseteq \mathcal P(\mathbf N)$ be a subadditive, primitive family for which there is $K \in \mathbf N$ such that $\rho_{k+1} \le \rho_k + K < \infty$ and $\lambda_k - K \le \lambda_{k+1}$ for all but finitely many $k$. Then the following hold:
\begin{enumerate}[label={\rm (\roman{*})}]
\item\label{it:th:structure-theorem-for-unions(i)} $\sup \Delta_\cup(\LLc) < \infty$.
\item\label{it:th:structure-theorem-for-unions(ii)} $\LLc$ satisfies the Structure Theorem for Unions.
\end{enumerate}
\end{theorem}
\begin{proof}
Both claims are trivial if $\Delta(\LLc)$ is empty, since this implies by Lemma \ref{prop:basic(1)}\ref{it:prop:basic(1)(ii)} that $\UUc_k \subseteq \{k\}$ for all $k$. So, we assume from now on that $\Delta(\LLc)$ is non-empty. Then $\delta \in \mathbf N^+$, and we obtain from Lemma \ref{lem:non-emptyness-of-multiples}\ref{it:lem:non-emptyness-of-multiples(ii)} that there exists $k^{\fixed[0.2]{\text{ }}\prime} \in \mathbf N$ such that $\UUc_k \ne \emptyset$ for $k \ge k^{\fixed[0.2]{\text{ }}\prime}$. Accordingly, we proceed as follows:

\ref{it:th:structure-theorem-for-unions(i)}
By hypothesis, there is $k^{\fixed[0.2]{\text{ }}\prime\prime} \in \mathbf N$ with the property that
\begin{equation}
\label{equ:bounded-apart}
\rho_{k+1} \le \rho_k + K < \infty
\quad\text{and}\quad
\lambda_k - K \le \lambda_{k+1}, \text{ for }k \ge k^{\fixed[0.2]{\text{ }}\prime\prime}.
\end{equation}
On the other hand, we know from Corollary \ref{cor:delta_sets}\ref{it:cor:delta_sets(iii)} that there exists $\ell \in \mathbf N^+$ with
\begin{equation}
\label{equ:keep-strong}
\ell + \delta \cdot \fixed[-0.1]{\text{ }} \llb 0, K \rrb \subseteq \UUc_\ell.
\end{equation}
Set $k_0 := \max(k^\prime, k^{\prime\prime})$.
By \eqref{equ:bounded-apart} and Lemma \ref{prop:basic(1)}\ref{it:prop:basic(1)(v)}, we have that
$\sup \Delta(\UUc_k) \le \rho_k < \infty$ for all $k$. So, it is sufficient to show that there exists $D \in \mathbf N$ such that $\sup \Delta(\UUc_k) \le D$ for all large $k$. To this end, let
$$
D := \max\bigl((1+K)\ell, {\max}_{0 \le i \le \ell-1} \sup \Delta({\UUc_{k_0+ i}})\bigr) \in \mathbf N^+.
$$
We will prove by (strong) induction that $\sup \Delta(\UUc_k) \le D$ for $k \ge k_0$.

If $k_0 \le k < k_0 + \ell$, the claim is obvious. Therefore, let $\kappa \ge k_0 + \ell$, and assume the conclusion is true for every $k \in \llb k_0, \kappa - 1 \rrb$. Since $\kappa - \ell \ge k_0$, $\UUc_{i}$ is non-empty, and hence $\lambda_{i} \in \mathbf N$, for every $i \in \llb \kappa - \ell, \kappa \rrb$.
In view of \eqref{equ:keep-strong} and Lemma \ref{prop:basic(1)}\ref{it:prop:basic(1)(iv)}, it follows that
\begin{equation}
\label{equ:make-it-easy}
\mathscr{V}_\kappa := \fixed[-0.2]{\text{ }} \bigl(\ell + \delta \cdot \fixed[-0.1]{\text{ }} \llb 0, K \rrb \bigl) \fixed[-0.2]{\text{ }} + \UUc_{\kappa - \ell} \subseteq \UUc_{\ell} + \UUc_{\kappa - \ell} \subseteq \UUc_{\kappa}.
\end{equation}
In addition, we have
\begin{equation}
\label{equ:min-max-of-Vk}
\sup \mathscr{V}_\kappa = \ell + \delta K + \rho_{\kappa - \ell} \in \UUc_\kappa
\quad\text{and}\quad
\inf \mathscr{V}_\kappa = \ell + \lambda_{\kappa - \ell} \in \UUc_\kappa.
\end{equation}
Consequently, we derive from \eqref{equ:bounded-apart} that
$$
0 \le \sup \UUc_\kappa - \sup \mathscr{V}_\kappa \le \rho_\kappa - \rho_{\kappa - \ell} = \sum_{i = \kappa - \ell}^{\kappa - 1} (\rho_{i+1} - \rho_i) \le  \ell K < D,
$$
and in a similar way,
$$
0 \le \inf \mathscr{V}_\kappa - \inf \UUc_\kappa = \ell + \lambda_{\kappa - \ell} - \lambda_\kappa = \ell + \sum_{i = \kappa - \ell}^{\kappa - 1} (\lambda_i - \lambda_{i+1}) \le (1 + K)\ell \le D.
$$
Thus, we are left to show that $\sup \Delta(\mathscr U_\kappa^\ast) \le D$, where
$$
\mathscr U_\kappa^\ast := \UUc_\kappa \cap \llb \ell + \lambda_{\kappa - \ell}, \ell + \delta K + \rho_{\kappa - \ell} \rrb.
$$
For, we obtain from \eqref{equ:make-it-easy} and \eqref{equ:min-max-of-Vk} that
$\mathscr V_\kappa \subseteq \mathscr U_\kappa^\ast$, $\sup \mathscr V_\kappa = \sup \mathscr U_\kappa^\ast$, and $
\inf \mathscr V_\kappa = \inf \mathscr U_\kappa^\ast$. Therefore, we see from Lemmas \ref{lem:delta-set-of-unions}\ref{it:lem:delta-set-of-unions(iii)} and \ref{lem:basic-properties-of-Delta} and the induction hypothesis, since $\kappa - \ell \in \llb k_0, \kappa - 1 \rrb$, that
$$
\sup \Delta(\mathscr U_\kappa^\ast) \le \sup \Delta(\mathscr V_\kappa) \le \sup \Delta(\UUc_{\kappa - \ell}) \le D.
$$
\ref{it:th:structure-theorem-for-unions(ii)} Let $r \in \mathbf N^+$. We infer from \eqref{equ:bounded-apart} that $(k + \delta \cdot \mathbf Z) \cap \llb \rho_{k-r}, \rho_k - (K+1)r\fixed[0.1]{\text{ }}  \rrb$ is empty for all but finitely many $k$, because $
\rho_k - \rho_{k-r} \le rK$ for $k \ge k^{\prime\prime}$ (cf. the proof of Corollary \ref{th:finite-delta-set-bound-on-growth-rate-of-rho_k}).
So, we conclude from \ref{it:th:structure-theorem-for-unions(i)} and Theorem \ref{th:characterization-of-structure-theorem} (applied with $N = K + 1$) that $\LLc$ satisfies the Structure Theorem for Unions.
\end{proof}
Finally, we combine some of the results obtained so far and establish a strong form of the Structure Theorem for Unions, valid for any subadditive family with accepted elasticity.
\begin{definition}
\label{def:strong-structure-theorem}
We say that a family $\mathscr L \subseteq \mathcal P(\mathbf N)$ satisfies the \textit{Strong Structure Theorem for Unions} if it satisfies the Structure Theorem for Unions and there exist $\mu, k_0 \in \mathbf N^+$ such that
$$\bigl((\mathscr U_k - \inf \mathscr U_k) \cap \llb 0, M \rrb\bigr)_{k \ge 1}
\quad\text{and}\quad
\bigl((\sup \mathscr U_k - \mathscr U_k) \cap \llb 0, M \rrb \bigr)_{k \ge 1}
$$
are $\mu$-periodic sequences for every $M \in \mathbf N$.
\end{definition}
We do not know whether there exists a subadditive, primitive subfamily of $\mathcal P(\mathbf N)$ with \textit{finite} elasticity that satisfies the Structure Theorem, but not the Strong Structure Theorem for Unions. However, on a positive note, the following holds:
\begin{theorem}
\label{th:summarizing-theorem}
Let $\LLc \subseteq \mathcal P(\mathbf N)$ be a subadditive, primitive family with accepted elasticity. Set $\delta^{\fixed[0.2]{\text{ }} \prime} := 1$ if $\Delta(\LLc) = \emptyset$ and $\delta^{\fixed[0.2]{\text{ }} \prime} := \delta$ otherwise.
Then $\LLc$ satisfies the Strong Structure Theorem for Unions.
\end{theorem}
\begin{proof}
If $\Delta(\LLc) = \emptyset$, we get from Lemmas \ref{prop:basic(1)}\ref{it:prop:basic(1)(ii)} and \ref{lem:non-emptyness-of-multiples}\ref{it:lem:non-emptyness-of-multiples(ii)} that $\UUc_k = \{k\}$ for all but finitely many $k$, so we can take $\mu := 1$ and the claim is trivial.
Thus we assume from now on that $\Delta(\LLc)$ is non-empty.

Then $\rho \ne 0$, and since $\LLc$ has accepted elasticity, we obtain from Proposition \ref{th:key-theorem_1} that there is $m \in \mathbf N^+$
such that
$
\rho_{k + m} \le \rho_{k} + m \rho$ and
$\lambda_{k + m} \ge \lambda_{k} - m\rho$ for all large $k$.
Consequently, we derive from Lemma \ref{lem:growth-rate-of-local-elasticities-in-weakly-directed-families} and
Theorems \ref{th:structure-theorem-for-unions} and \ref{cor:main-corollary} that $\mathscr L$ satisfies the Strong Structure Theorem for Unions.
\end{proof}
We conclude the section with a corollary generalizing \cite[Corollary 2.3(1)]{FGKT}. To this end, we say that a set $L \subseteq \mathbf N$ is an \textit{almost arithmetic progression} (shortly, AAP) \textit{with difference $d$ and bound $M$}, for some $d \in \mathbf N^+$ and $M \in \mathbf N$, if there exists $z \in \mathbf Z$ such that
$$
(z + d \cdot \mathbf Z) \cap \llb \inf L + M, \sup L - M \rrb \subseteq L \subseteq z + d \cdot \bf Z,
$$
see \cite[Definition 4.2.1]{GeHK06} for an equivalent, though slightly different, definition.
\begin{corollary}
\label{cor:summarizing}
Let $\LLc \subseteq \mathcal P(\mathbf N)$ be a subadditive family satisfying the Structure Theorem for Unions, and assume $\rho_k < \infty$ for every $k \in \mathbf N$. Then there is $M \in \mathbf N$ such that $\UUc_k$ is an \textup{AAP} with difference $\delta^{\fixed[0.2]{\text{ }} \prime}$ and bound $M$ for all $k \in \mathbf N$, where $\delta^{\fixed[0.2]{\text{ }} \prime} := 1$ if $\Delta(\LLc) = \emptyset$ and $\delta^{\fixed[0.2]{\text{ }} \prime} := \delta$ otherwise.
\end{corollary}
\begin{proof}
If $\Delta(\LLc)$ is empty, Lemma \ref{prop:basic(1)}\ref{it:prop:basic(1)(ii)} yields $\UUc_{k} \subseteq \{k\}$ for all $k$, and the claim is trivial. Otherwise, it follows from our assumptions and Proposition \ref{prop:the-difference-is-necessarily-delta} that there exist $k_0, M \in \mathbf N$ such that, for $k \ge k_0$, $\UUc_k$ is an AAP with difference $\delta$ and bound $M$. Since $\rho_k < \infty$ for every $k$, this, in turn, implies that $\UUc_0, \UUc_1, \ldots$ are all AAPs with difference $\delta$ and bound $\max(M, N)$, where $N := 1+\max(\rho_0, \ldots, \rho_{k_0-1})$.
\end{proof}
\section{A focus on systems of sets of lengths}
\label{sec:focus-on-sets-of-lengths}
In this short section, we apply the main results of \S{ }\ref{sec:basic-properties-of-directed-families} to the structure of unions of sets of lengths of a monoid. We start with a proof of the theorems stated in \S{ }\ref{sec:intro} (we will freely use notations and terminology from the introduction and Examples \ref{exa:systems-of-sets-of-lengths} and \ref{exa:system-of-sets-of-lengths}).
\begin{proof}[Proof of Theorems \textup{\ref{th:structure-theorem}} and \textup{\ref{th:main-theorem-intro}}]
We know from Example \ref{exa:system-of-sets-of-lengths} that $\mathscr L(H)$ is a directed subfamily of $\mathcal P(\mathbf N)$, unless the set of atoms of $H$ is empty, in which case $\mathscr L(H) = \fixed[-0.2]{\text{ }} \bigl\{\{0\}\bigr\}$. Moreover, it is clear that $\Delta(H) = \Delta(\mathscr L(H))$ and $\UUc_k(H) = \UUc_k(\mathscr L(H))$ for all $k$, and that $H$ has accepted elasticity if and only if so does $\mathscr L(H)$. This is enough to conclude the proof, by applying Theorems \ref{th:structure-theorem-for-unions}\ref{it:th:structure-theorem-for-unions(ii)} and \ref{th:summarizing-theorem} to $\mathscr L(H)$, and by noticing that every directed subfamily of $\mathcal P(\mathbf N)$ is primitive.
\end{proof}
The next step is a characterization of when a monoid satisfies the Structure Theorem for Unions:
\begin{theorem}
\label{th:characterization-of-structure-theorem-for-monoids}
Let $H$ be a monoid, and set $\delta^{\fixed[0.2]{\text{ }} \prime} := 1$ if $\Delta(H) = \emptyset$ and $\delta^{\fixed[0.2]{\text{ }} \prime} := \inf \Delta(H)$ otherwise. Then $H$ satisfies the Structure Theorem for Unions if and only if there exist $D, N \in \mathbf N^+$ such that, for all large $k$, the following conditions hold:
\begin{enumerate*}[label={\rm (\roman{*})}] \item\label{it:th:characterization-of-structure-theorem-for-monoids(i)}
	$\sup \Delta(\UUc_k(H)) \le D$;	
\item\label{it:th:characterization-of-structure-theorem-for-monoids(ii)}
	$(k + \delta \cdot \mathbf Z) \cap \llb \rho_{k - \ell} + \ell, \rho_k - N \rrb \subseteq \UUc_k$, where $\ell$ is any positive integer with the property that $\ell + \delta \cdot \llb 0, D \rrb \subseteq \UUc_\ell(H)$.
\end{enumerate*}
\end{theorem}
\begin{proof}
If $\Delta(H) = \emptyset$, the conclusion is obvious, since $\UUc_k(H) \subseteq \{k\}$ for all $k$. Otherwise, the claim follows by Theorem \ref{th:characterization-of-structure-theorem} and the same considerations as in the above proof of Theorems \ref{th:structure-theorem} and \ref{th:main-theorem-intro}.
\end{proof}
A variety of monoids (and domains) satisfying conditions \ref{it:th:characterization-of-structure-theorem-for-monoids(i)} and \ref{it:th:characterization-of-structure-theorem-for-monoids(ii)} of Theorem \ref{th:characterization-of-structure-theorem-for-monoids}, and hence the Structure Theorem for Unions, can be found in \cite[\S{ }3]{FGKT}: In this regard, note that, by Remark \ref{remark:finiteness-of-delta-set-implies-finiteness-of-Udelta-set}, condition \ref{it:th:characterization-of-structure-theorem-for-monoids(i)} is implied by the finiteness of $\Delta(H)$, as we have already observed that $\Delta(H) = \Delta(\LLc(H))$.

So from here on we restrict our attention to Theorem \ref{th:main-theorem-intro}: The goal is to identify some interesting classes of monoids with accepted elasticity. To this end, we need a few more definitions.

\begin{definition}
Let $H$ and $K$ be (multiplicatively written) monoids, and let $\varphi$ be a (monoid) homomorphism $H \to K$.
We call $\varphi$ \textit{essentially surjective} if $K = K^\times \varphi(H) K^\times$, and an \textit{equimorphism} if:
\begin{enumerate}[label={\rm (\textsc{e}\arabic{*})}]
\item\label{covariant-transfer(1)} $\varphi^{-1}(K^\times) \subseteq H^\times$ (or equivalently $\varphi^{-1}(K^\times) = H^\times$).
\item\label{covariant-transfer(2)} $\varphi$ is \textit{atom-preserving}, i.e., $\varphi(a) \in \mathcal A(K)$ for all $a \in \mathcal A(H)$.
\item\label{covariant-transfer(3)} If $x \in H$ and $\varphi(x) = b_1 \cdots b_n$ for some $b_1, \ldots, b_n \in \mathcal A(K)$, then there exist $\sigma \in \mathfrak S_n$ and $a_1, \ldots, a_n \in \mathcal A(H)$ such that $x = a_1 \cdots a_n$ and $b_{\sigma(i)} \simeq_K \varphi(a_i)$ for every $i \in \llb 1, n \rrb$.
\end{enumerate}
We say that $H$ is \textit{essentially equimorphic} to $K$ if there is an essentially surjective equimorphism from $H$ to $K$; and a \textit{transfer Krull monoid of finite type} if $H$ is essentially equimorphic to a monoid of zero-sum sequences over an abelian group $G$ with support in a finite set $G_0 \subseteq G$.
\end{definition}
We refer to \cite[Remarks 2.17--2.20]{FTr} for a critical comparison of these definitions with analogous ones from the literature on factorization theory: In particular, a weak transfer ho\-mo\-mor\-phism in the sense of
\cite[Definition 2.1]{BaSm}
is an essentially surjective equimorphism, by \cite[Remark 2.19]{FTr}.

The interest here in equimorphisms stems from the next proposition, which provides sufficient conditions for a monoid to have accepted elasticity that are often met in practice (see below for examples), and where a monoid $H$ is said to satisfy the Strong Structure Theorem for Unions if so does $\mathscr L(H)$.
\begin{theorem}
\label{prop:fundamental-properties-of-dense-equis}
Let $\varphi: H \to K$ an essentially surjective equimorphism. The following hold:
\begin{enumerate}[label={\rm (\roman{*})}]
\item\label{it:prop:fundamental-properties-of-dense-equis(i)}
For every $y \in K \setminus K^\times$ there exists $x \in H \setminus H^\times$ with $y \simeq_K \varphi(x)$ and $\mathsf L_H(x) = \mathsf L_K(y)$.
\item\label{it:prop:fundamental-properties-of-dense-equis(ii)} $\mathscr{L}(H) = \mathscr{L}(K)$.
\item\label{it:prop:fundamental-properties-of-dense-equis(iii)} If $K$ is a cancellative, commutative monoid and the quotient $K/K^\times$ is finitely generated, then $H$ has accepted elasticity and satisfies the Strong Structure Theorem for Unions.
\end{enumerate}
\end{theorem}
\begin{proof}
\ref{it:prop:fundamental-properties-of-dense-equis(i)} Pick $y \in K \setminus K^\times$. Since $\varphi$ is essentially surjective, $y = u \fixed[0.2]{\text{ }} \varphi(x) \fixed[0.2]{\text{ }} v$ for some $x \in H$ and $u, v \in K^\times$. Accordingly, \cite[Lemma 2.2(iv) and Theorem 2.22(i)]{FTr} yield $\mathsf L_K(y) = {\sf L}_K(\varphi(x)) = {\sf L}_H(x)$. Moreover, $x$ is not a unit of $H$, otherwise $y = u \fixed[0.2]{\text{ }} \varphi(x) \fixed[0.2]{\text{ }} v \in K^\times$ (because units are preserved under homomorphisms).

\ref{it:prop:fundamental-properties-of-dense-equis(ii)} We know from \cite[Theorem 2.22]{FTr} that $\mathscr{L}(H) \subseteq \mathscr{L}(K)$, and we have by \ref{it:prop:fundamental-properties-of-dense-equis(i)} that $\mathscr{L}(K) \subseteq \mathscr{L}(H)$.

\ref{it:prop:fundamental-properties-of-dense-equis(iii)} Since $H$ is essentially equimorphic to $K$, we get from \ref{it:prop:fundamental-properties-of-dense-equis(ii)} that $H$ and $K$ have the same system of sets of lengths, and hence $\rho(\mathscr L(H)) = \rho(\mathscr L(K))$. This shows that $H$ has accepted elasticity, because the assumptions on $K$ imply, by \cite[Theorem 3.1.4]{GeHK06}, that $\rho(\mathscr L(K)) = \rho(L) < \infty$ for some $L \in \mathscr L(K)$. The rest is a consequence of Theorem \ref{th:main-theorem-intro}.
\end{proof}
Now we provide a short list of monoids (and domains) with accepted elasticity: By Theorem \ref{th:main-theorem-intro}, all of them satisfy the Strong Structure Theorem for Unions.
\begin{examples}
\label{exa:short-list}
(1) Transfer Krull monoids of finite type, as we get from our definitions and Theorem \ref{prop:fundamental-properties-of-dense-equis}\ref{it:prop:fundamental-properties-of-dense-equis(iii)}: This is a fairly large, important class of monoids, which contains (among others):
    \begin{enumerate}[label={\rm (\roman{*})}]
    \item All Krull monoids with finite class group, see \cite[Theorems 3.4.10]{GeHK06}, and hence the multiplicative monoid of non-zero elements of any commutative Dedekind domain with finite class group.
    \item Every classical maximal $\mathbf Z_K$-order $R$ in a central simple algebra over a number field $K$ such that all stably free left $R$-ideals are free (here, $\mathbf Z_K$ denotes the ring of integers of $K$), as we infer from a much more comprehensive result of D.~Smertnig on classical maximal orders over holomorphy rings in global fields,
        see \cite[Theorem 1.1]{Sm13}.
    \end{enumerate}
    For further examples along the same lines, see \cite[\S{ }4, pp. 977--978]{Ge16c} and references therein.
\vskip 0.1cm
\noindent{}(2) Every $v$-Noetherian weakly Krull commutative monoid $H$ with non-empty conductor $(H: \widehat{H})$ and finite elasticity such that the $v$-class group of $H$ is finite and the localization of $H$ at $\mathfrak p$ is finitely primary for any minimal prime ideal $\mathfrak p$ of $H$ {(see \cite{GeHK06} for notations and terminology),} as implied by \cite[Theorem 4.4]{GeZh17}: Remarkably, this class includes all orders in number fields with finite elasticity, and the finiteness of the elasticity is equivalent to the bijectivity of the canonical map $\pi: {\rm spec}(\widehat{H}) \to {\rm spec}(H): \mathfrak p \mapsto \mathfrak p \cap H$.
\vskip 0.1cm
\noindent{}(3) All numerical monoids, viz., submonoids $H$ of $(\mathbf N, +)$ with $|\mathbf N \setminus H| < \infty$:
For one thing, these are not transfer Krull monoids of finite type unless they are equal to $(\mathbf N, +)$, as we obtain from \cite[Theorem 5.5.2]{GeSchZh17}. But they are cancellative, finitely generated, commutative, and reduced (i.e., the group of units is trivial), and hence have accepted elasticity by Theorem \ref{prop:fundamental-properties-of-dense-equis}\ref{it:prop:fundamental-properties-of-dense-equis(iii)}.
%
\vskip 0.1cm
\noindent{}(4) Some local arithmetical congruence monoids \cite[Theorem 1.1]{Pono}, where an arithmetical congruence monoid is a submonoid of the multiplicative monoid of $\mathbf N$ of the form $\{1\} \cup (a + b \cdot \mathbf N)$ with $a, b \in \mathbf N^+$ and $a^2 \equiv a \bmod b$, and is called local if $\gcd(a,b) = p^r$ for some prime $p$ and $r \in \mathbf N$.
\vskip 0.1cm
\noindent{}(5) All Puiseux monoids (that is, submonoids of the non-negative rational numbers under addition) whose set of atoms has both a maximum and a minimum, see \cite[Theorem 3.4]{GoONe17}.
\end{examples}
To finish, we give an example, due to Alfred Geroldinger, of a Dedekind domain whose multiplicative monoid has accepted elasticity and infinite set of distances (cf. Remark \ref{remark:finiteness-of-delta-set-implies-finiteness-of-Udelta-set}).
\begin{example}
\label{exa:infinite-set-of-distances}
We get from \cite[Proposition 4.1.2.5]{GeHK06} that, for all $n, r \in \mathbf N^+$ with $2 \le n \ne r+1$, there are an abelian group $H$ and a finite set $H_0 \subseteq H$ for which
$$
\Delta(\mathcal B(H_0)) = \fixed[-0.2]{\text{ }} \bigl\{|n-r-1|\bigr\}
\quad\text{and}\quad
\rho(\mathcal B(H_0)) = \max\fixed[-1]{\text{ }} \left(\frac{n}{r+1}, \frac{r+1}{n}\right)\fixed[-0.6]{\text{ }},
$$
where $\mathcal B(H_0)$ denotes the monoid of zero-sum sequences over $H$ with support in $H_0$ (see Example \ref{exa:systems-of-sets-of-lengths}). In particular, since $|H_0| < \infty$, we find by \cite[Theorem 3.4.2.1]{GeHK06} that $\mathcal B(H_0)$ is a reduced, finitely generated, commutative, cancellative monoid, and hence has accepted elasticity by Theorem \ref{prop:fundamental-properties-of-dense-equis}\ref{it:prop:fundamental-properties-of-dense-equis(iii)}.

It follows that, for every $k \ge 1$, there are an abelian group $G_k$ and a set $G_k^{\fixed[0.2]{\text{ }}\prime} \subseteq G_k$ such that $\Delta(\mathcal B(G_k^{\fixed[0.2]{\text{ }}\prime})) = \{k\}$, $\rho(\mathcal B(G_k^{\fixed[0.2]{\text{ }}\prime})) = 2$, and $\mathcal B(G_k^{\fixed[0.2]{\text{ }}\prime})$ has accepted elasticity (take $r = 2k+1$ and $n = k+1$ in the above construction).
Accordingly, let $G$ be the direct sum of the groups $G_1, G_2, \ldots$, and $G_0 \subseteq G$ the disjoint union of the sets $G_1^{\fixed[0.2]{\text{ }}\prime}, G_2^{\fixed[0.2]{\text{ }}\prime}, \ldots$ It is then seen that $\mathcal B(G_0)$ is the coproduct of the monoids $\mathcal B(G_1^{\fixed[0.2]{\text{ }}\prime}), \mathcal B(G_2^{\fixed[0.2]{\text{ }}\prime}), \ldots$, which shows by \cite[Proposition 1.4.5]{GeHK06} that $\Delta(\mathcal B(G_0)) = \bigcup_{k \ge 1} \Delta(\mathcal B(G_k^{\fixed[0.2]{\text{ }}\prime})) = \mathbf N^+$ and  $\mathcal B(G_0)$ has accepted elasticity (therefore, $\mathcal B(G_0)$ satisfies the Strong Structure Theorem for Unions, by Theorem \ref{th:main-theorem-intro}).

So, by Claborn's Realization Theorem (see, e.g., \cite[Theorem 3.7.8]{GeHK06}), there exist a Dedekind domain $R$ with class group $\mathcal C(R)$ and a group isomorphism $\varphi: G \to \mathcal C(R)$ such that $\varphi(G_0)$ is the set, $G_P$, of all ideal classes of $R$ containing prime ideals, with the result that $\mathscr L(\mathcal B(G_0)) = \mathscr L(\mathcal B(G_P))$.

With this in hand, let $R^\bullet$ be the monoid of non-zero elements of $R$ under multiplication. We have by \cite[Example 2.3.2.1]{GeHK06} that $R^\bullet$ is a Krull monoid (recall that every Dedekind domain is a Krull domain). Therefore, we conclude from \cite[Theorem 3.4.10]{GeHK06} that $
\mathscr L(R^\bullet) = \mathscr L(\mathcal B(G_P)) = \mathscr L(\mathcal B(G_0))$,
which implies that (the multiplicative monoid of) $R$ satisfies the Strong Structure Theorem for Unions and $\Delta(R) = \mathbf N^+$.
\end{example}
\section*{Acknowledgements}
\label{subsec:acks}
The author is grateful to Alfred Geroldinger for asking the basic questions that have inspired this work and, more in general, for his guidance through the kaleidoscopic lands of factorization theory.

\end{document}